\newif\ifpreprint\preprinttrue 
\ifpreprint
  \documentclass[10pt,reqno]{amsart}
\else
  \documentclass[final]{svjour3}
  \journalname{Foundations of Computational Mathematics}
\fi
\usepackage{amssymb}	
\usepackage{tikz}
\usepackage{tikz-cd}
\usepackage{amsmath}
\usepackage{mathrsfs}
\usepackage{verbatim}
\usepackage{graphicx}
\usepackage{color}
\usepackage{float}
\usepackage{bm}
\usepackage{amsfonts}
\usepackage{amscd}
\usepackage{cases}
\usepackage{xcolor}
\usepackage{mathabx}
\usepackage{hyperref}

\newcommand\tr{\operatorname{tr}}
\newcommand\sskw{\operatorname{sskw}}
\newcommand\inc{\operatorname{inc}}
\newcommand\skw{\operatorname{skw}}
\newcommand\vskw{\operatorname{vskw}}
\newcommand\mskw{\operatorname{mskw}}

\newcommand\sym{\operatorname{sym}}
\newcommand\grad{\operatorname{grad}}
\newcommand\alt{\operatorname{Alt}}
\renewcommand\div{\operatorname{div}}
\renewcommand\ker{\mathcal{N}}

\newcommand\curl{\operatorname{curl}}
\newcommand\rot{\operatorname{rot}}

\newcommand\dev{\operatorname{dev}}

\newcommand\spn{\operatorname{span}}
\newcommand\ran{\mathcal{R}}

\newcommand\K{\mathbb{K}}
\newcommand\M{\mathbb{M}}
\newcommand\cinc{\operatorname{cinc}}
\newcommand\dd{\,\mathrm{d}}
\renewcommand\S{{\mathbb S}}

\newcommand\R{\mathbb{R}}

\newcommand\x{\times}

\newcommand\V{{\mathbb{V}}}
\newcommand\W{{\mathbb{W}}}
\newcommand\Z{{\mathbb{Z}}}

\newcommand\deff{\operatorname{def}}
\renewcommand\gg{\operatorname{hess}}

\renewcommand\dd{{\div\div}}
\newcommand\sign{\operatorname{sgn}}

\newcommand{\bs}{{\scriptscriptstyle \bullet}}

\ifpreprint
\newtheorem{theorem}{Theorem}
\newtheorem{remark}{Remark}

\newtheorem{definition}{Definition}
\newtheorem{corollary}{Corollary}
\newtheorem{lemma}{Lemma}

\newtheorem{proposition}{Proposition}

\fi

\begin{document}

\ifpreprint
\title{Complexes from complexes}
\author{Douglas N. Arnold}
\address{Department of Mathematics, University of Minnesota, Minneapolis, MN, USA}
\email{arnold@umn.edu}
\author{Kaibo Hu}
\address{Department of Mathematics, University of Minnesota, Minneapolis, MN, USA}
\email{khu@umn.edu}
\thanks{The work of the first author was supported supported by NSF grant DMS-1719694 and Simons Foundation grant 601937, DNA}
\fi

\begin{abstract}
This paper is concerned with the derivation and properties of differential complexes arising from a variety of problems in differential equations, with applications in continuum mechanics, relativity, and other fields.  We present a systematic procedure which, starting from well-understood differential complexes such as the de Rham complex, derives new complexes and deduces the properties of the new complexes from the old.  We relate the cohomology of the output complex to that of the input complexes and show that the new complex has closed ranges, and, consequently, satisfies a Hodge decomposition, Poincar\'e type inequalities, well-posed Hodge-Laplacian boundary value problems, regular decomposition, and compactness properties on general Lipschitz domains.
\end{abstract}

\ifpreprint\maketitle\fi

\section{Introduction}

Differential complexes are an important tool in the modeling, analysis, and---increasingly---the numerics of a number of problems. In physics, differential complexes relate to the decomposition of a field into a potential and a complementary part.  Recently, we have come to understand the extent to which stability and convergence of numerical methods rely on the preservation of the underlying structures of the differential complexes, in particular, the cohomology. Building on early works on finite element differential forms \cite{Bossavit.A.1998a,Hiptmair.R.2002a}, this point of view has been developed into the framework of the finite element exterior calculus (FEEC) by Arnold, Falk and Winther \cite{Arnold.D;Falk.R;Winther.R.2006a,Arnold.D;Falk.R;Winther.R.2010a} among others. 
  
The most canonical differential complex is the de~Rham complex.  It is of fundamental importance in numerous applications, such as electromagnetism and fluid dynamics, and by now it is vastly studied.  But there are many other complexes that arise in different applications and relate to different differential equations, the best known perhaps being the elasticity complex.  Key functional analytic and regularity properties of these other complexes are crucial for analysis and numerics, but are not so well understood. In this paper, we present a systematic procedure which, starting from well-understood differential complexes, constructs new complexes and deduces the properties of the new complexes from the old.

In order to better clarify the contents of the paper, we now quickly review the de~Rham complex on a bounded Lipschitz domain $\Omega\subset\R^n$, in several variant forms.  This discussion will be elaborated in Section~\ref{sec:preliminary} of the paper. The basic homological structure is captured
in the \emph{smooth} de~Rham complex, in which the spaces consist of differential forms with smooth coefficients (note that $\Lambda^{k}$ in the notation indicates a space of differential $k$-forms) and the differentials are exterior derivatives:
\begin{equation}\label{Cinfty-deRham}
\begin{tikzcd}
0 \arrow{r} &C^{\infty}\Lambda^{0} \arrow{r}{d^{0}} &C^{\infty}\Lambda^{1} \arrow{r}{d^{1}} &  \cdots \arrow{r}{d^{n-1}} & C^{\infty}\Lambda^{n}\arrow{r}&0.
 \end{tikzcd}
\end{equation}
If we restrict to three dimensions, we can use scalar and vector proxies to write this in calculus notation:
\begin{equation*}
\begin{tikzcd}
0 \arrow{r} &C^{\infty}(\Omega) \arrow{r}{\grad} &C^{\infty}(\Omega;\R^3) \arrow{r}{\curl} & 
C^{\infty}(\Omega;\R^3) \arrow{r}{\div} & C^{\infty}(\Omega)\arrow{r}&0.
 \end{tikzcd}
\end{equation*}
Here $C^{\infty}(\Omega)$ is the usual space of all infinitely differentiable functions on $\Omega$.

Additional analytical properties are captured in a variant of the smooth complex, namely the \emph{Sobolev}
de~Rham complex, which extends the exterior derivatives to less regular differential forms and encodes the fact that they are operators of first order. For any real number $q$, it reads: 
\begin{equation}\label{sobolev-deRham}
\begin{tikzcd}
0 \arrow{r} &H^{q}\Lambda^{0} \arrow{r}{d^{0}} &H^{q-1}\Lambda^{1} \arrow{r}{d^{1}} &  \cdots \arrow{r}{d^{n-1}} & H^{q-n}\Lambda^{n}\arrow{r}&0.
 \end{tikzcd}
\end{equation}
This is a \emph{bounded Hilbert complex}, meaning that the spaces are Hilbert spaces and the operators bounded linear operators.
Further functional analytic structure is encoded in another variant, the
$L^{2}$ de~Rham Hilbert complex where the differential operators are not bounded, but merely closed and densely defined,
\begin{equation}\label{L2-deRham}
\begin{tikzcd}
0 \arrow{r} &L^{2}\Lambda^{0} \arrow{r}{d^{0}} &L^{2}\Lambda^{1} \arrow{r}{d^{1}} &  \cdots \arrow{r}{d^{n-1}} & L^{2}\Lambda^{n}\arrow{r}&0.
 \end{tikzcd}
\end{equation}
 Their domains are defined to be the spaces
$$
H\Lambda^{k}=\left \{ u\in L^{2}\Lambda^{k}: ~d^{k}u\in L^{2}\Lambda^{k+1}\right \}.
$$
Restricting to the domains furnishes yet another variant of the de~Rham complex, a bounded Hilbert complex called the \emph{domain complex} of the $L^2$ de~Rham complex:
\begin{equation}\label{HD-deRham}
\begin{tikzcd}
0 \arrow{r} &H\Lambda^{0} \arrow{r}{d^{0}} &H\Lambda^{1} \arrow{r}{d^{1}} &  \cdots \arrow{r}{d^{n-1}} & H\Lambda^{n}\arrow{r}&0.
 \end{tikzcd}
\end{equation}

The $L^2$ de~Rham complex plays a crucial role in FEEC, and we refer to \cite{arnold2018finite,Arnold.D;Falk.R;Winther.R.2006a} and the references therein for many results related to it.
The Sobolev de~Rham complex  \eqref{sobolev-deRham} and generalizations of it were studied extensively by Costabel and McIntosh in \cite{costabel2010bogovskiui} under rather weak assumptions on the regularity of the domain.  From their results, we may obtain numerous fundamental properties of the de~Rham complex:
\begin{itemize}
\item The complexes \eqref{sobolev-deRham} and \eqref{L2-deRham} are \emph{closed} in the sense that
all the differential operators have closed range.  This is a crucial assumption of the FEEC framework, which implies the Poincar\'e inequality, the Hodge decomposition, and well-posedness of the Hodge Laplacian boundary value problem, among other results.
\item For each of the variant complexes above, the cohomology spaces are finite dimensional and mutually isomorphic.  A common single set of $C^\infty$ cohomology representatives can be chosen.
\item If the domain $\Omega$ is contractible, then each of the complexes has vanishing cohomology except at the level $0$, where the cohomology space is $\R$.  In other words, each complex is locally a resolution
of the constants.
\item Each space in \eqref{HD-deRham} admits a regular decomposition.
\item  The complexes \eqref{L2-deRham} satisfy a compactness property.
\end{itemize}
We will define and discuss these properties in greater detail in Section~\ref{sec:preliminary} below. 

As mentioned above, many problems arising in continuum mechanics and differential geometry require other, more complicated differential complexes, the best known being the \emph{elasticity complex}, also called the {\it Kr\"{o}ner complex} in mechanics or the {\it (linearized) Calabi complex} in geometry. In three space dimensions, the smooth elasticity complex reads
 \begin{equation}\label{sequence:3Delasticity}
\begin{tikzcd}
0 \arrow{r}&  C^{\infty}\otimes \mathbb{V} \arrow{r}{{\deff}} &C^{\infty}\otimes  \mathbb{S}  \arrow{r}{\inc} &   C^{\infty}\otimes \mathbb{S} \arrow{r}{\div} & C^{\infty}\otimes \mathbb{V}\arrow{r}&0.
 \end{tikzcd}
\end{equation}
It is locally a resolution of the rigid motions. Here we write $\mathbb{V}$ for the space $\R^3$ of $3$-vectors
and so $C^{\infty}\otimes\mathbb{V}=C^{\infty}(\Omega)\otimes \mathbb{V}$ is the space of smooth vector fields.  Similarly, $\mathbb{S}$ and $C^{\infty}\otimes \mathbb{S}$ denote the spaces of $3\x3$ symmetric matrices and smooth matrix fields, respectively.
The operators in the elasticity complex are the \emph{deformation} or linearized strain operator $\operatorname{def}=\operatorname{sym}\grad$, the \emph{incompatibility} operator $\inc =\curl\circ \mathrm{T}\circ\curl$ (where $\mathrm{T}$ denotes the transpose operation and $\curl$ acts on a matrix field by rows), and the (row-wise) \emph{divergence} operator operating on matrix fields.  Note that the incompatibility operator is second-order.

The elasticity complex has been crucial to the development of mixed finite element methods for elasticity \cite{arnold2002mixed,arnold2007mixed}.  The incompatibility operator $\inc$ appears in the Saint-Venant condition $\inc e=0$ giving the conditions for a symmetric matrix field $e$ to locally equal the deformation (strain tensor) $\deff u$ of a displacement vector field. It is further utilized in the development of intrinsic elasticity \cite{ciarlet2009intrinsic} where the deformation field replaces the displacement field as the primary unknown. The incompatibility operator is also central to Kr\"{o}ner's pioneering work on dislocation theory \cite{kroner1975dislocations,van2010non}, where $\inc$ applied to a strain tensor measures the density of dislocations. Its application to problems such as elastoplasticity remains an active research area \cite{amstutz2017incompatibility,amstutz2018incompatibility,geymonat2005some}.  The analogy between the operators of the de~Rham complex and those of the elasticity complex has been long noted, going back at least to \cite[Table 1]{seeger1961recent} and \cite{kroner1975dislocations}.

Other complexes combining first and second order differentials arise in other applications, particularly the \emph{Hessian complex} 
 \begin{equation}\label{sequence:hessian}
\begin{tikzcd}
0 \arrow{r}&  C^{\infty} \arrow{r}{{\gg}} &C^{\infty}\otimes  \mathbb{S}  \arrow{r}{\curl} &   C^{\infty}\otimes \mathbb{T} \arrow{r}{\div} & C^{\infty}\otimes \mathbb{V}\arrow{r}&0,
 \end{tikzcd}
\end{equation}
and its formal adjoint, the \emph{$\dd$ complex,}
 \begin{equation}\label{sequence:divdiv}
\begin{tikzcd}[row sep=large]
0 \arrow{r}&  C^{\infty}\otimes \mathbb{V} \arrow{r}{{\dev\grad}} &[11] C^{\infty}\otimes  \mathbb{T}  \arrow{r}{\sym\curl} &[11]  C^{\infty}\otimes \mathbb{S} \arrow{r}{\dd} &[11] C^{\infty}\arrow{r}&0.
 \end{tikzcd}
\end{equation}
Here $\mathbb{T}$ is the space of trace-free matrices, and $\dev$ is the deviatoric operator
which sends a matrix to its trace-free part.  These complexes have been used for plate and other biharmonic problems by Pauly and Zulehner \cite{pauly2016closed,pauly2018divdiv} and have been applied to the Einstein equations by Quenneville-B{\'e}lair \cite{quenneville2015new}.

While the elasticity, Hessian, $\dd$, and other complexes have important applications, there has not been a systematic investigation of their derivations or fundamental properties. For example, the crucial closed range property required to fit the complexes into the FEEC framework has not yet been established in general, nor has the independence of the cohomology on the Sobolev regularity (although a variety of special cases and partial results have appeared \cite{amstutz2016analysis,angoshtari2015differential,ciarlet2013linear,geymonat2006beltrami,horgan1995korn,pauly2016closed,pauly2018divdiv}).
In this paper, we present a systematic way to obtain and analyze such complexes via an algebraic construction presented in Section~\ref{sec:framework} which derives new complexes from existing ones.  The construction is related to the Bernstein--Gelfand--Gelfand (BGG) resolution from the representation of Lie algebras \cite{eastwood1999variations,eastwood2000complex,vcap2001bernstein}, but we shall not rely on that, and instead provide a self-contained presentation.  In Section~\ref{sec:framework}, we present the derivation of the new complex together with two key theorems relating its cohomology to that of the input complexes.  The proofs of these theorems are postponed to Section~\ref{sec:proof}, but first, in Section~\ref{sec:main}, we apply the results of Section~\ref{sec:framework} in numerous ways to obtain a variety of
complexes (elasticity, Hessian, $\dd$,  $\grad\curl$, $\curl\div$, $\grad\div$, conformal elasticity, and conformal Hessian) with a variety of applications. We emphasize that the value of this paper lies not only in the numerous results obtained for numerous complexes in Section~\ref{sec:main}, but also in the systematic approach to obtaining these results from known results for the de~Rham complex and similar complexes.  We believe this BGG-based construction will prove valuable in other contexts, both to extend to other complexes and to obtain additional properties.  An example in this direction is in \cite{christiansen2019poincare} where the BGG approach is used to construct Poincar\'e operators for the elasticity complex from classical Poincar\'e operators based on path integrals for the de~Rham complex.

This paper is focused on the construction and analysis of differential
complexes which relate to important PDE problems from continuum mechanics
and other applications.  However, another important motivation for the work is to
enable the development of stable and accurate \emph{discretization methods} for
solving these problems.  An important conclusion of the finite element exterior
calculus is that a stable finite element method for a problem arising from
a differential complex requires finite element spaces that form a subcomplex
of the original complex and admit a cochain projection from the complex on
the continuous level to that on the discrete level.  The construction of such
finite element spaces has been systematically investigated and achieved for
the de Rham complex.  For the elasticity complex, it was achieved in 2002
when, after attempts going back four decades, the first stable mixed finite
elements for elasticity with polynomial shape functions were discovered
in two dimensions \cite{arnold2002mixed}.  In that work, and particularly in the follow-up work
in \cite{arnold2006a} and \cite{arnold2007mixed}, the construction of finite elements for the elasticity
complex was guided by the corresponding derivation of the elasticity complex
at the continuous level from the de Rham complex, together with the use
of known stable finite element discretizations of the de Rham complex.
This approach has been followed by numerous authors since, such as another
discretization of the 2D elasticity complex obtained  by Christiansen, Hu,
and Hu \cite{christiansen2016nodal} by combining a discrete Stokes complex and an Hermite
finite element discretization of the de~Rham complex.  The current paper
develops the systematic derivation of new complexes from known complexes,
with the derivation of the elasticity complex from the de Rham complex being
one example of many.  Consequently our results should provide guidance for
the development of finite element discretization of these new complexes,
providing stable finite element methods to solve numerous problems for which
they were heretofore unavailable.

We close the introduction by noting that the approach of this paper provides a new way to prove important analytical results such as Korn's inequality.   Korn's inequality is nothing other than the first Poincar\'{e} inequality associated with the elasticity complex, and so follows from the closed range property of that complex, which is established here as a consequence of known properties of the de~Rham complex together with homological algebra.  A stronger, but lesser known inequality, the trace-free Korn's inequality, fits into the same framework and is proved in a similar way, but for a different complex. Cf., Section~\ref{sec:conformal}. Similar observations apply to other operators, such as $\inc$, furnishing more or less familiar inequalities.

\section{Sobolev scales of complexes and their properties}\label{sec:preliminary}

Before continuing, we make precise some terminology and notation.  Suppose we have a two vector spaces
$V$ and $W$ and a linear operator $D$ mapping between them.  Sometimes we will allow for the case
where $D$ is defined only on a subspace of $V$, called its domain, rather than on the whole space.
The kernel of $D$, which we denote by $\ker(D)$ or $\ker(D,V)$, is nevertheless a well-defined
subspace of $V$, and the range of $D$, $\ran(D)$ or $\ran(D,V)$, is a subspace of $W$.
Now suppose that we have a sequence of
vector spaces and linear operators mapping one to the next,
$$
\cdots \to Z^{k-1} \xrightarrow{D^{k-1}} Z^k \xrightarrow{D^k} Z^{k+1} \to \cdots,
$$
allowing for the case in which the domain of $D^k$ is a proper subspace of $Z^k$.
The sequence is called a \emph{complex} if $\ran(D^{k-1})\subset \ker(D^k)$ for each $k$.  Then we can
define the $k$th cohomology space $\mathscr{H}^k$ as the vector space $\ker(D^k)/\ran(D^{k-1})$.  The complex is called a \emph{Hilbert complex}
if the spaces $Z^k$ are Hilbert spaces and the operators $D^k$ are closed and densely defined \cite{bru1992hilbert,glotko2003complex,Arnold.D;Falk.R;Winther.R.2010a,arnold2018finite}.
Note that, for a Hilbert complex, the nullspace $\ker(D^k)$, being the kernel of a closed operator, is a closed subspace of $Z^k$, but the range $\ran(D^{k-1})$ need not be closed.  Therefore the cohomology space $\mathscr{H}^k$
may not be a Hilbert space. In this case it does not coincide with the Hilbert space $\bar{\mathscr{H}}^k:=\ker(D^{k})/\overline{\ran(D^{k-1})}$, which is called the \emph{reduced cohomology space}. If each
$\ran(D^k)$ is closed, we say the Hilbert complex is closed, and then the distinction between cohomology and reduced
cohomology disappears.

Throughout this paper, we assume that $\Omega$ is a bounded Lipschitz domain in $\R^n$.  We shall consider Sobolev spaces of functions (or distributions) on $\Omega$ taking values in a finite dimensional Hilbert space $\mathbb{E}$ (for example, we might have $\mathbb{E} = \R^n$).  We may identify such a vector-valued Sobolev space with a tensor product, and so denote it by $H^{q}(\Omega)\otimes \mathbb{E}$, or just $H^{q}\otimes \mathbb{E}$, where $q\in\R$ is the order of the Sobolev space.  This is a Hilbert space, whose norm we denote by $\|\,\cdot\,\|_{q}$.  In the case $q=0$, i.e., the space $L^2\otimes\mathbb{E}$, we may just write $\|\,\cdot\,\|$. For a (possibly unbounded) linear operator $\mathscr{D}$ which maps from one such $L^2$ space to another, we may use the graph norm given by $\|u\|_{\mathscr{D}}^{2}:=\|u\|^{2}+\|\mathscr{D}u\|^{2}$.  We write $d^k$ for the exterior derivative operator from $k$-forms to $(k+1)$-forms (it vanishes for $k<0$ or $k>n-1$).

The Sobolev--de~Rham complex \eqref{sobolev-deRham} depends on the Sobolev order $q$, and so is actually a scale of complexes, by which we mean a family of complexes
\begin{equation}\label{scale}
\cdots\to Z^{k}_{[q]} \xrightarrow{D^{k}_{[q]}} Z^{k+1}_{[q]} \to \cdots
\end{equation}
indexed by a parameter $q\in\R$, such that if $q'\ge q$, then
$$
Z^{k}_{[q']}\subset Z^{k}_{[q]}
\text{\quad and\quad}
D^{k}_{[q']}=D^{k}_{[q]}|_{Z^{k}_{[q']}}.
$$
In this paper we will derive numerous scales of complexes of the form
\begin{equation}\label{sobscale}
\cdots\to H^{q_k}\otimes \W^k \xrightarrow{  D^{k}} H^{q_{k+1}}\otimes \W^{k+1} \to \cdots,
\end{equation}
The spaces
are vector-valued Sobolev spaces of the form $Z^k_{[q]}=H^{q_k}\otimes \W^k$ where the $\W^k$ are
finite dimensional inner product spaces and the differentials $  D^k$ are linear
differential operators of some positive real order $\gamma_k\geq 1$.  The real numbers $q_k$ are given by $q_0=q$ and $q_{k+1}=q_k-\gamma_k$.

For such a Sobolev scale of complexes there is an $L^2$ Hilbert complex variant, just as for the de~Rham complex.
The complex is
\begin{equation}\label{L2W}
\cdots \to L^2\otimes\W^k \xrightarrow{  D^{k}}L^2\otimes\W^{k+1} \to \cdots,
\end{equation}
where now $D^k$ is a differential operator with constant (or more generally, smooth) coefficients defined in the sense of distributions and viewed as a closed unbounded operator with domain
$$
H\W^k:= \{\, u\in L^2\otimes\W^k\,|\,   D^ku \in L^2\otimes\W^{k+1}\,\}.
$$
This operator is indeed closed (c.f., \cite[Section 6.2.6]{arnold2018finite}). It is densely defined because $C^{\infty}_0(\Omega)\otimes \mathbb{W}^{k}$ is dense in $L^2\otimes\mathbb{W}^k$. This leads to the following $L^{2}$ domain complex 
\begin{equation}\label{domainW}
\cdots \to H\W^k \xrightarrow{  D^{k}}H\W^{k+1} \to \cdots.
\end{equation}

In many important cases, at each level $k$, the cohomology of the complexes in the scale can be represented by a single  set of smooth functions, independent of $q$.

\begin{definition}\label{urc}
A sequence of finite-dimensional spaces $G^k_\infty\subset L^2(\Omega)\otimes \W^k$ is said to \emph{uniformly represent
the cohomology} of a scale of complexes \eqref{scale} if, for each $k\in\Z$ and each $q\in\R$,
\begin{align}\label{NRG}
\ker(  D^k,Z^k_{[q]}) = \ran(  D^{k-1},Z^{k-1}_{[q]}) \oplus G^k_\infty.
\end{align}
\end{definition}
Note that, in case the scale of complexes is of the form \eqref{sobscale}, then, by definition, the space $G^k_\infty$ belongs to  all the Sobolev spaces $H^q\otimes \W^k$, so it is contained
in $C^\infty$. In the rest of this discussion we assume that there exists a uniform representation of the cohomology for the Sobolev scale \eqref{sobscale}.

Almost all the complexes we treat will be closed (recall that this means that the range space $\ran(D^{k-1}, Z^{k-1})$ is closed in $Z^k$ for each $k$). In particular if the cohomology is finite dimensional, then the range space is closed (\cite[Lemma 19.1.1]{hormander1994analysis}). 
Moreover, if a scale of Sobolev complexes has a uniform representation of cohomology, then, for each $q$, the complex has finite dimensional cohomology and so is closed. 

Since $H^{\gamma_{k-1}}\otimes\W^{k}\subset H\W^{k}\subset L^{2}\otimes\W^{k}$, the cohomology of the $L^{2}$ complex \eqref{L2W} can be represented by the same representatives as the Sobolev complex \eqref{sobscale}. Specifically, we have the following result.
\begin{theorem}\label{thm:elasticity-Hilbert-complex}
Suppose that the scale of complexes \eqref{sobscale} admits a uniform set of cohomology representatives $G^k_\infty$.  Then the same spaces are cohomology representatives for the domain complex \eqref{domainW} as well:
$$
\ker\left ( {D}^{k}, H\W^{k}\right )=\ran({D}^{k-1}, H\W^{k-1})\oplus {G}_{\infty}^{k}.
$$
\end{theorem}
\begin{proof}
We have 
\begin{align*}
\ker\left ( {D}^{k}, H\W^{k}\right )=\ker&\left ( {D}^{k}, L^{2}\otimes \W^{k}\right )=\ran({D}^{k-1}, H^{\gamma_{k-1}}\otimes \W^{k-1}  )\oplus {G}_{\infty}^{k}\\&
\subset \ran({D}^{k-1}, H\W^{k-1} )\oplus {G}_{\infty}^{k}\subset \ker\left ( {D}^{k}, H\W^{k}\right ),
\end{align*}
where the first equality is by definition. This implies the result.
\end{proof}

From the fact that the complex is closed we may derive numerous consequences.

\subsubsection*{Hodge decomposition}
An important consequence of the closedness of the range of $D$ is the Hodge decomposition. 
Let $  D_{k}^{\ast}: L^{2}\otimes\W^{k}\to L^{2}\otimes\W^{k-1}$ be the adjoint operator of the unbounded operator $  D^{k-1}: L^{2}\otimes\W^{k-1}\to L^{2}\otimes\W^{k}$ associated with \eqref{L2W}. We denote the  domain of the adjoint by $H^{\ast}\W^{k}$.Recall that for any densely defined linear operator between Hilbert spaces $T: X\to Y$, the adjoint of $T$ is defined to be an unbounded operator with the domain
$$
D(T^{\ast}):=\{w\in Y: \exists c_{w}>0 ~s.t.~ |\langle w, Tv\rangle_{Y}|\leq c_{w}\|v\|_{X}, ~\forall v\in D(T)\}.
$$ 
In specific examples, it consists of forms $u\in L^{2}\otimes\W^{k}$ for which the formal adjoint $  D_k^{\ast} u\in L^{2}\otimes\W^{k-1}$ and which satisfy certain boundary conditions.  See \cite[Theorem 6.5]{arnold2018finite} for the case of the de~Rham complex.
With this notation, the Hodge decomposition is easily derived.  Let
$$
\mathfrak{H}^{k}= \{\,u\in \ker(  D^{k})\,|\, u\perp\ran(  D^{k-1}) \,\}
$$
denote the harmonic forms for this Hilbert complex.  Then we have
\begin{equation}\label{Hodge}
L^2\otimes\W^k = \ker(  D^k) \oplus \ker(  D^k)^\perp
=  \ker(  D^k) \oplus \ran(  D^*_{k+1}) = \ran(  D^{k-1})\oplus \mathfrak{H}^{k} \oplus \ran(  D^*_{k+1}),
\end{equation}
which is the Hodge decomposition in this context.  Besides the definitions, we have used duality and the closed range theorem, which ensure that the orthogonal complement of the kernel of an operator with closed range coincides with the range of its adjoint.

\subsubsection*{Poincar\'{e} inequality}
Another important consequence of the closed range property is the Poincar\'{e} inequality, from which it follows by Banach's theorem. For \eqref{L2W}, the Poincar\'{e} inequality reads: 
$$
\left\|u\right\|\lesssim \left\| {D}^{k}u\right\|, \quad \forall u\in H\W^{k}, \, u\perp_{L^{2}} \ker\left (  {D}^{k}, H\W^{k}\right ),
$$
and for \eqref{sobolev-deRham}:
\begin{align}\label{poincare-Hs}
\left\|u\right\|_{q_{k}}\lesssim \left\| {D}^{k}u\right\|_{q_{k+1}}, \quad \forall u\in H^{q_{k}}\otimes \W^{k}, \, u\perp_{H^{q_{k}}} \ker\left ( {D}^{k}, H^{q}\otimes \W^{k}\right ).
\end{align}
We write $a\lesssim b$ to mean $a\leq Cb$ for some generic constant $C$.

\subsubsection*{Well-posed Hodge-Laplacian boundary value problem}
The Hodge decomposition and Poincar\'{e} inequalities then imply the well-posedness of the Hodge-Laplacian boundary value problems associated with the $L^{2}$ domain complex \eqref{domainW}, up to harmonic forms. We refer to \cite[Section~4.4.2]{arnold2018finite} for the proof.

\bigskip

The preceding properties were all deduced from the fact that the cohomology of the $L^{2}$ domain complex \eqref{domainW} 
is finite dimensional.  Now we present three more important properties that require in addition the existence of a uniform representation of  cohomology of \eqref{sobscale}.

\subsubsection*{Existence of regular potentials}

\begin{theorem}[Existence of bounded regular potentials]\label{thm:regpot}
Let $q,r\in\R$, $k\in\Z$.  There is a constant $C$ such that
for any $v\in H^q\otimes\W^{k+1}\cap \ran (  D^{k}, H^r\otimes\W^{k})$, there exists $u\in H^{q+\gamma_{k}}\otimes \W^{k}$ such that $ {D}^{k}u=v$ and
\begin{align}\label{bounded-potential}
\|u\|_{q+\gamma_{k}}\leq C\|v\|_{q}. 
\end{align}
\end{theorem}
\begin{proof}
By the assumptions on $v$, it belongs to $\ker(D^{k+1}, H^{q}\otimes \W^{k+1})$.  Now, from the uniform representation of cohomology applied to the sequence
\begin{equation*}
H^{q+\gamma_{k}}\otimes \W^{k} \xrightarrow{D^k} H^q\otimes \W^{k+1} \xrightarrow{D^{k+1}} H^{q-\gamma_{k+1}}\otimes \W^{k+2},
\end{equation*}
we have $\ker(D^{k+1}, H^{q}\otimes \W^{k+1}) = \ran(D^k,H^{q+\gamma_{k}}\otimes \W^{k}) + G^{k+1}_\infty$, so there exists $u\in H^{q+\gamma_{k}}\otimes \W^{k}$ and $s\in G^{k+1}_\infty$ for which $v=Du+s$.  But $v\in \ran( {D}^k, H^r\otimes\W^{k})$ and $Du\in \ran( {D}^k, H^{q+\gamma_k}\otimes\W^{k})$, so $s\in G^{k+1}_\infty \cap \ran( {D}^k, H^t\otimes\W^{k})$ with $t=\min(q+\gamma_k,r)$.  But the last space reduces to zero, since the sum in \eqref{NRG} is direct. Further, we may subtract from $u$ its projection onto $\ker(D^k, H^{q+\gamma_{k}}\otimes \W^{k})$ without changing $D^ku$.
Then $u$ is the desired regular potential and the bound \eqref{bounded-potential} is immediate from Poincar\'e inequality \eqref{poincare-Hs}. 
\end{proof}

\subsubsection*{Regular decomposition} 
The regular decomposition of the de~Rham complex and its discrete version have various applications in numerical analysis, see, e.g., \cite{hiptmair2017discrete} and the references therein.
A classical proof of the regular decomposition relies on the Fourier analysis and extensions of vector fields \cite{hiptmair2012universal}. However, we now show that regular decompositions for the more general complex \eqref{domainW} can be deduced directly from the Hodge decomposition of the $L^{2}$ complex \eqref{L2W} and the existence of regular potentials. 

\begin{theorem}\label{thm:regular}
The regular decomposition holds:
\begin{align}\label{regular-decomposition}
H\W^{k}=  {D}^{k-1}\left ( H^{\gamma_{k-1}}\otimes \W^{k-1}\right )+H^{\gamma_{k}}\otimes \W^{k}.
\end{align}
\end{theorem}
\begin{proof}
Let $w\in H \W^{k}$. Applying Theorem~\ref{thm:regpot} to $v= {D}^{k}w$, we obtain $u\in H^{\gamma_{k}}\otimes \W^{k}$ such that $ {D}^{k}u= {D}^{k} w$. Then $w-u\in\ker(D^k, L^2\otimes\W^k)$, so, by uniform representation of cohomology, there exists $y\in H^{\gamma_{k-1}}\otimes \W^{k-1}$ and $s\in G^{k}_\infty\otimes \W^{k}$ such that $w - u= {D}^{k-1}y +s$. Then $w= {D}^{k-1}y+(u + s)$ provides a regular decomposition for $u$.
\end{proof}

\subsubsection*{Compactness property}
The space $H\W^{k}\cap H^{\ast}\W^{k}$, i.e., the intersection of the domains of $D$ and $D^*$,
is a Hilbert space with the norm $u\mapsto (\|u^{n}\|+\|{D}^{k}u^{n}\|+\|{D}^{\ast}_{k}u^{n}\|)^{1/2}$.
Its inclusion into $L^2\otimes\W^k$ is obviously continuous.  The compactness property states that the inclusion is in fact compact.
\begin{theorem}\label{thm:compactness}
The imbedding $H\W^{k}\cap H^{\ast}\W^{k}\hookrightarrow L^{2}\otimes \W^{k}$ is compact. 
\end{theorem}
The classical proof of the compactness property for the de~Rham complex is due to Picard \cite{picard1984elementary}. Here we provide a proof for general complex \eqref{L2W} based on the existence of regular potentials and the classical Rellich compactness theorem for $H^{1}$ scalar functions.  A similar proof can be found in \cite[Lemma 3.19]{pauly2016closed}.

\begin{proof}
Let $\{u^{n}\}$ be a bounded sequence in $H \W^{k}\cap H^{\ast} \W^{k}$, so $\|u^{n}\|+\|{D}^{k}u^{n}\|+\|{D}^{\ast}_{k}u^{n}\|$ is bounded. We must show that there exists a convergent subsequence in $L^{2}\otimes \W^{k}$.  Expanding $u^n$ by the Hodge decomposition, we have
\begin{equation}\label{hd}
u^n = v^n + w^n + h^n,
\end{equation}
where the sequences $\{v^n\}$, $\{w^n\}$, and $\{h^n\}$ belong to the spaces $\ran(D^{k-1}, H\W^{k-1})$, $\ran(D^*_{k+1}, H^*\W^{k+1})$, and $\mathfrak H^k$, respectively, and are $L^2$ bounded.  We shall show that each of these sequences admits a convergent subsequence, giving the theorem.  This is
certainly true for the $h^n$ sequence, since $\dim\mathfrak H^k<\infty$.

To show that $\{v^n\}$ has a convergent subsequence we introduce a regular potential $y^n\in H^{\gamma_{k-1}}\otimes\W^{k-1}$ with
$Dy^n = v^n$ and with the $y^n$ uniformly bounded in $H^{\gamma_{k-1}}$.  We can then apply the Rellich compactness theorem to obtain 
a subsequence, which we continue to denote $y^n$, which converges in $L^2$. Now, from \eqref{hd}, we see that $v^n\in H^*\W^k$ and $D^*v^n =
D^*u^n$ is $L^2$-bounded uniformly in $n$.  Thus
$$
\left\|v^{m}-v^{n}\right\|^{2}= ({D}^{k-1}(y^m-y^n),  v^{m}-v^{n} )= (y^m-y^n,  {D}^*(v^{m}-v^{n}) )
$$
which tends to zero as $m,n\to\infty$, since $\{y^n\}$ is Cauchy in $L^2$ and $\{D^*v^n\}$ is $L^2$-bounded.

By a completely analogous argument applied to the dual Hilbert complex (which is also closed), we find a convergent subsequence of $\{w^n\}$, and so complete the proof.
\end{proof}

\begin{remark}
From the compactness results and the other properties, one can derive generalized div-curl lemmas which may be applied to nonlinear problems. Cf.~Pauly \cite{pauly2019global}. 
\end{remark}

In this section we have seen that if a Sobolev scale of complexes in the form \eqref{sobscale} admits a uniform representation of cohomology in the sense of Definition~\ref{urc}, then it possesses all the numerous properties discussed above.  As a primary example we have the de~Rham complex.
In \cite{costabel2010bogovskiui}, Costabel and McIntosh investigated the Sobolev de~Rham complex \eqref{sobolev-deRham} on general Lipschitz domains and established the uniform representation of cohomology. Their primary tools were regularized  path integrals of Poincar{\'e} and Bogovski{\u\i}, which provide a contracting homotopy
of the exterior derivatives, which they showed are pseudo\-differential operators of order $-1$.
\begin{theorem}[Costabel and McIntosh]\label{thm:CM}
On any bounded Lipschitz domain in $\R^{n}$ and for any real number $q$, the cohomology of the Sobolev de~Rham complex \eqref{sobolev-deRham} has finite dimension independent of $q$. Moreover, the cohomology can be represented by smooth functions, again independent of $q$.  In other words, there exists a finite-dimensional space $H_{\infty}^{k}\subset C^{\infty}\Lambda^{k}$ such that
\begin{equation}\label{uniform-deRham}
\ker(d^{k}, H^q\Lambda^k)=\ran(d^{k-1}, H^{q+1}\Lambda^{k-1})\oplus H_{\infty}^{k}, \quad  q\in\R,\ 0\leq k\leq n.
\end{equation}
\end{theorem}

From this theorem and the arguments in this section we obtain a new proof of the fundamental properties of the de~Rham complex which is alternative to more classical arguments, cf.~\cite{arnold2018finite,Arnold.D;Falk.R;Winther.R.2006a} and the references therein).

Finally, we remark that we have stated Theorem~\ref{thm:CM} for the $L^2$ based Sobolev spaces, but it was proven in 
\cite{costabel2010bogovskiui} also for a variety of Banach, Besov, and Triebel--Lizorkin spaces and a number
of the results of this section would extend to these.

\section{Algebraic construction of a complex and its cohomology}\label{sec:framework}

In this section we present the algebraic construction by which we derive new differential complexes from known ones, and then we relate the cohomology of the output complex to that of the input complexes. We carry this out in an abstract setting. For simplicity, we restrict to complexes of Hilbert spaces, although some of the results could be generalized to Banach spaces without major changes.

We start with two bounded Hilbert complexes $(Z^\bs,D^\bs)$, $(\tilde Z^\bs,\tilde D^\bs)$ and bounded linking
maps $S^i: \tilde{Z}^{i}\to Z^{i+1}, ~i=-1, \cdots, n$: 
\begin{equation}\label{Z-complex}
\begin{tikzcd}
0 \arrow{r} &Z^{0} \arrow{r}{D^{0}} &Z^{1} \arrow{r}{D^{1}} &\cdots \arrow{r}{D^{n-1}} & Z^{n}  \arrow{r}{} & 0\\
0 \arrow{r} &\tilde{Z}^{0}\arrow{r}{\tilde{D}^{0}} \arrow[ur, "S^{0}"]&\tilde{Z}^{1} \arrow{r}{\tilde{D}^{1}} \arrow[ur, "S^{1}"]&\cdots \arrow{r}{\tilde{D}^{n-1}}\arrow[ur, "S^{n-1}"] & \tilde{Z}^{n} \arrow{r}{} & 0
 \end{tikzcd}
\end{equation}
(the zero maps $S^{-1}$ and $S^{n}$ are not shown).
This means that the spaces $Z^{i}$ and $\tilde{Z}^{i}$, $i=0, 1, \cdots, n$, are Hilbert spaces and the maps $D^{i}$, $\tilde{D}^{i}$, $i=0, 1, \cdots, n-1$, are bounded linear operators. 
The two complexes in \eqref{Z-complex} cannot be arbitrary. Instead,  we require that the spaces be of the form
\begin{equation}\label{zform}
Z^{i}:=V^{i}\otimes \mathbb{E}^{i} \quad \mbox{and} \quad \tilde{Z}^{i}:=V^{i+1}\otimes \tilde{\mathbb{E}}^{i}
\end{equation}
for given Hilbert spaces $V^{i}$ and finite dimensional inner product spaces  $\mathbb{E}^{i}$ and $\tilde{\mathbb{E}}^{i}$.
In typical applications, $V^{i}$ is a Sobolev space and $\mathbb{E}^{i}, \tilde{\mathbb{E}}^{i}$ might be the space of scalars, vectors, matrices, symmetric matrices, trace-free matrices, or skew symmetric matrices (denoted by $\R$, $\mathbb{V}$, $\mathbb{M}$, $\mathbb{S}$, $\mathbb{T}$, and $\mathbb{K}$, respectively, and equipped with the Frobenius norm).

In addition, we assume that the connecting operators $S^i$ are of the form
\begin{equation}\label{sform}
S^{i}=\mathrm{id}\otimes {s}^{i}
\end{equation}
where ${s}^{i}: \tilde{\mathbb{E}}^{i}\to \mathbb{E}^{i+1}$ is a linear operator between finite dimensional spaces for which we require two key properties,
\begin{itemize}
\item
 \emph{Anticommutativity:}
\begin{align}\label{SDDS}
S^{i+1}\tilde{D}^{i}=-D^{i+1}S^{i}, \quad  i=0, 1, \cdots, n-2,
\end{align}
\item The \emph{$J$-injectivity/surjectivity condition}: for some particular $J$ with $0\le J < n$,
\begin{equation}\label{injsurj}
  \text{$s^{i}$ is }
  \begin{cases}\text{injective}, & 0\leq i\leq J,\\\text{surjective}, & J\leq i < n.\end{cases}
\end{equation}
\end{itemize}  Note that the latter condition implies that $s^{J}$ is bijective.

From the $s^i$ maps we obtain the null spaces  $\ker(s^{i})\subset \tilde{\mathbb{E}}^{i}$ and the ranges: $\ran(s^{i-1})\subset \mathbb{E}^{i}$.
With all these ingredients, we now define the \emph{output complex} (in Theorem~\ref{thm:dimension} below we show that it is indeed a complex):
\begin{equation}\label{reduced-complex}
\begin{tikzcd}
0 \to\Upsilon^{0}\arrow{r}{\mathscr{D}^{0}} &\Upsilon^{1} \arrow{r}{\mathscr{D}^{1}} &\cdots \arrow{r}{\mathscr{D}^{J-1}}&\Upsilon^{J} \arrow{r}{\mathscr{D}^{J}}  &\Upsilon^{J+1} \arrow{r}{\mathscr{D}^{J+1}} &\cdots\arrow{r}{\mathscr{D}^{n-1}} &\Upsilon^{n}\to 0,
 \end{tikzcd}
\end{equation}
with spaces
\begin{equation}\label{defUps}
\Upsilon^{i}:=
\begin{cases}
V^{i}\otimes \ran ({s}^{i-1} )^{\perp}, \quad 0\leq i\leq J,\\
V^{i+1}\otimes  \ker ({s}^{i} ), \quad J<i\leq n,
\end{cases}
\end{equation}
and operators
\begin{equation}\label{Dder}
\mathscr{D}^{i}=
\begin{cases}
(\mathrm{id}\otimes P_{\ran^{\perp}} )D^{i}, \quad i<J;\\
 \tilde{D}^{J}(S^{J})^{-1}D^{J}, \quad i=J;\\
\tilde{D}^{i}, \quad i>J.
\end{cases}
\end{equation}
Here for a closed subspace $\mathbb{F}$ of some Hilbert space, we write ${P}_{\mathbb{F}}$ for the orthogonal projection onto $\mathbb{F}$ and $\mathbb{F}^\perp$ for its orthogonal complement. Of course $P_{\mathbb{F}^\perp}  = I - P_{\mathbb{F}}$. We will be particularly interested in the case where the subspace is $\ran(s^{i-1})\subset \mathbb{E}^i$ or $\ker(s^i)\subset \tilde{\mathbb{E}}^i$.  When confusion is unlikely we shorten the notation to ${P}_{\ker}$, $P_{\ker^{\perp}}$, $P_{\ran}$, and $P_{\ran^{\perp}}$ for ${P}_{\ker(s^{i})}$, etc. With a slight abuse of notation, we also denote the projections in $Z^{i}=V^i\otimes \mathbb{E}^i$ and $\tilde{Z}^{i}=V^{i+1}\otimes \tilde{\mathbb{E}}^i$, i.e.,  $\mathrm{id}\otimes {P}_{\ker}$, $\mathrm{id}\otimes P_{\ker^{\perp}}$, $\mathrm{id}\otimes P_{\ran}$, and $\mathrm{id}\otimes P_{\ran^{\perp}}$ by ${P}_{\ker}$, $P_{\ker^{\perp}}$, $P_{\ran}$, and $P_{\ran^{\perp}}$, respectively.

Note that $\mathscr{D}^{i}$ maps $\Upsilon^{i}$ to $\Upsilon^{i+1}$ for $i<J$ because we included the orthogonal projection onto $\ran(s^i)^\perp$ in its definition, while for $i\ge J$, $\mathscr{D}^i$  maps $\Upsilon^{i}$ to $\Upsilon^{i+1}$ due to the anticommutativity.

We can read out the output complex from the input $Z$ and $\tilde{Z}$ complexes as follows. We start from the left end of the top row of \eqref{Z-complex} where the $S$ operators are injective, and follow the complex rightwards, at each step restricting to the orthogonal complement of the ranges of the incoming $S$ operators. When we reach the space $Z^J$ we map to the space $\tilde Z^{J+1}$ in the bottom row by following a zig-zag path, rightwards
into  $Z^{J+1}$ by $D^J$, then down and to the left into $\tilde Z^J$ by following the linking map $S^J$ in the reverse direction (which is possible since it is a bijection), and then rightwards into $\tilde Z^{J+1}$ by $\tilde D^J$.  We then continue
rightwards along the bottom complex, restricting  to the kernels of the $S$ operators.

This completes the construction of the output complex \eqref{reduced-complex} in the abstract setting. In Section~\ref{sec:main}, we will apply it to derive the elasticity complex, the Hessian complex, and the $\dd$ complex in $3$ dimensions, generalizations of these to $n$ dimensions, and other complexes. 
In order to establish the properties of the output complex, a key result is the relation between its cohomology and that of the two input complexes \eqref{Z-complex}.  This is described in the following two theorems, which are main results of this paper. Theorem~\ref{thm:dimension} verifies that the output complex is indeed a bounded Hilbert complex and relates the dimensions of its cohomology spaces to those of the input complexes.  Under an additional assumption, Theorem~\ref{thm:rep} gives an explicit map between the output and input complexes and show that it induces an isomophism  on the cohomology. The proofs of these results will be given in Section~\ref{sec:proof}.

\begin{theorem}\label{thm:dimension}
Let there be given bounded Hilbert complexes $(Z^\bs,D^\bs)$ and $(\tilde Z^\bs,\tilde D^\bs)$ and bounded linking
maps $S^i: \tilde{Z}^{i}\to Z^{i+1}$ satisfying \eqref{Z-complex}--\eqref{injsurj}. Then the output complex defined by \eqref{reduced-complex}--\eqref{Dder} is a bounded Hilbert complex.  Moreover,
$$
\dim \mathscr{H}^{i}\left ( \Upsilon^{\bs}, \mathscr{D}^{\bs}\right )\leq \dim \mathscr{H}^{i}\left ( Z^{\bs}, {D}^{\bs}\right )+\dim \mathscr{H}^{i} ( \tilde{Z}^{\bs}, \tilde{D}^{\bs} ), \quad \forall i=0, 1, \cdots, n
$$
(where $\mathscr{H}^i$ denotes the $i$th cohomology space).
Finally, equality holds if and only if $S^i$ induces the zero maps on cohomology, i.e., if and only if
\begin{align}\label{SNR}
S^{i}\ker(\tilde{D}^{i})\subset \ran ( D^{i}), \quad \forall i=0, 1, \cdots, n-1.
\end{align}
\end{theorem}

We immediately obtain from this theorem that, if the input complexes in  \eqref{Z-complex} have finite dimensional cohomology, then so does the output complex. This, in turn, implies that the operators in the output complex have closed range, and the numerous properties that this implies, as explained in Section~\ref{sec:preliminary}.


It is easy to verify the cohomology condition \eqref{SNR}, if we assume that there exist bounded operators $K^{i}: \tilde{Z}^{i}\rightarrow Z^{i}$, $i=0, 1, \ldots, n$, such that
\begin{equation}\label{SDKKD}
S^{i}=D^{i}K^{i}-K^{i+1}\tilde{D}^{i},\quad i=0,1,\ldots,n-1.
\end{equation}

\begin{proposition}\label{prop:K}
Assume that there exist bounded operators $K^{i}$ satisfying \eqref{SDKKD}. Then \eqref{SNR} holds.
\end{proposition}
\begin{proof}
$S^{i}\ker(\tilde{D}^{i})=(D^{i}K^{i}-K^{i+1}\tilde{D}^{i})\ker (\tilde{D}^{i})=D^{i}(K^{i}\ker (\tilde{D}^{i}))\subset \ran(D^{i})$.
\end{proof}

Under the same assumption, we can give an explicit isomorphism from the Cartesian product of the cohomology spaces of the $Z$ and $\tilde{Z}$ complexes to the cohomology spaces of the output complex.  
Clearly, $s^{i-1}:\tilde{\mathbb{E}}^{i-1}\to \mathbb{E}^i$ restricts to an isomorphism of $\ker(s^{i-1})^\perp$ onto $\ran(s^{i-1})$. Its Moore--Penrose inverse, which we denote by $t^i:\mathbb{E}^i\to \tilde{\mathbb{E}}^{i-1}$, is defined to act as the inverse of this isomorphism on $\ran(s^{i-1})$ and to vanish on its orthogonal complement.  Equivalently, the compositions of $t^i$ and $s^i$ are the orthogonal projections
\begin{equation}\label{tsst}
t^i s^{i-1}=P_{\ker(s^{i-1})^\perp},\quad  s^{i-1} t^i=P_{\ran(s^{i-1})}. 
\end{equation}
We also let $T^i=\mathrm{id}\otimes t^i$.
The proof of the following theorem will be given in Section \ref{sec:proof}. 

\begin{theorem}\label{thm:rep}
Assume that there exist operators $K^i$ satisfying \eqref{SDKKD}.  Define
$\mathscr{K^i}:Z^i\x \tilde Z^i\to \Upsilon^i$ by
\begin{equation}\label{def:Kmap}
\mathscr{K}^{i}(\omega,\mu) = \begin{cases} P_{\ran(s^{i-1})^{\perp}}(\omega+K^{i}\mu), & 0\le i\le J,\\
P_{\ker(s^{i})}[\tilde{D}^{i-1}T^{i}\omega+(I+\tilde{D}^{i-1}T^{i}K^i)\mu], & J< i\le n.
\end{cases}
\end{equation}
This defines a cochain map from the sum complex $Z^\bs\x \tilde Z^\bs$ to the output complex $\Upsilon^\bs$
for which the induced map on cohomology is an isomorphism.
\end{theorem}

As an immediate corollary of Theorem \ref{thm:rep}, we have an explicit representation of the cohomology of the output complex. 
\begin{corollary}\label{cor:rep}
Assume that $H^{\bs}$ and $\tilde{H}^{\bs}$ are cohomology representatives of the $Z$ and $\tilde{Z}$ complexes, respectively, i.e.,
$$
\ker(D^{i}, Z^{i})=\ran(D^{i-1}, Z^{i-1})\oplus H^{i}, \quad 1\leq i\leq n,
$$
and 
$$
\ker(\tilde{D}^{i}, \tilde{Z}^{i})=\ran(\tilde{D}^{i-1}, \tilde{Z}^{i-1})\oplus \tilde{H}^{i}, \quad 1\leq i\leq n.
$$
Then 
\begin{equation}\label{explicit-cohomology}
\ker (\mathscr{D}^{i}, \Upsilon^{i})=\ran(\mathscr{D}^{i-1}, \Upsilon^{i-1})\oplus 
\begin{cases}
P_{\ran(s^{i-1})^{\perp}}(H^{i}+K^{i}\tilde{H}^{i}), \quad 1\leq i\leq J,\\
P_{\ker(s^{i})}[\tilde{D}^{i-1}T^{i}H^{i}+(I+\tilde{D}^{i-1}T^{i}K^i)\tilde{H}^{i}], & J< i\le n.
\end{cases}
\end{equation}
\end{corollary}

\section{Applications}\label{sec:main}

In order to apply the algebraic construction from the last section we must specify
the Hilbert spaces $V^{i}$, finite dimensional inner product spaces $\mathbb{E}^{i}$ and $\tilde{\mathbb{E}}^{i}$, and the linking maps $s^{i}: \tilde{\mathbb{E}}^{i}\to \mathbb{E}^{i+1}$, and we must verify the anticommutativity
 property and the $J$-surjectivity/injectivity condition (for a particular $J$).  This then furnishes an output complex
\eqref{reduced-complex} satisfying Theorem~\ref{thm:dimension}.

In Section~\ref{sec:three-D-example} we show how to derive the elasticity, Hessian, and div-div complexes in
three dimensions, beginning with variants of the de~Rham complex.  This example is then generalized to $n$-dimensions
in Section~\ref{sec:first-example}.  Additional complexes are derived in Sections~\ref{sec:iterated}
and \ref{sec:conformal}.

\subsection{Applications in three dimensions using vector proxies}\label{sec:three-D-example}

We begin with some elementary examples, using vector calculus notation.
To this end, we introduce notations for some basic linear algebraic operations in $\R^n$:
\begin{itemize}
\item $\skw: \M\to \K$ and $\sym:\M\to\S$ are the skew and symmetric part operators,
\item $\tr:\M\to\R$ is the matrix trace,
\item $\iota: \R\to \M$ is the map $\iota u:= uI$ identifying a scalar with a scalar matrix,
\item $\dev:\mathbb{M}\to \mathbb{T}$ given by $\dev u:=u-1/n \tr (u)I$ is the deviator, or trace-free part.
\end{itemize}
In three dimensions only, we also have an isomorphism between skew symmetric matrices and vectors defined by
the map
$$
 \mskw\left ( 
\begin{array}{c}
v_{1}\\ v_{2}\\ v_{3}
\end{array}
\right ):= \left ( 
\begin{array}{ccc}
0 & -v_{3} & v_{2} \\
v_{3} & 0 & -v_{1}\\
-v_{2} & v_{1} & 0
\end{array}
\right ),
$$
Thus the operator $\mskw$ maps $\V$ isomorphically onto $\K$ and satisfies $\mskw(v)w = v\times w$ for $v,w\in\V$. The vector $v$ is said to be the axial vector of the skew matrix $\mskw(v)$.  We also define
$\vskw=\mskw^{-1}\circ \skw: \M\to \V$, the map taking a matrix to the axial vector of its skew symmetric part.  Finally, we define the map $S:\M\to\M$ by $Su = u^T-\tr(u)I$.  This map is
invertible in any number of dimensions $n>1$.

Now, let $\Omega$ be a Lipschitz domain in $\R^3$ and $q$ any real number and consider the following diagram
whose rows are complexes joined by linking maps \cite{arnold2015beijing}:
 \begin{equation}\label{diagram-4rows}
\begin{tikzcd}
0 \arrow{r} &H^{q}\otimes \mathbb{R}  \arrow{r}{\grad} &H^{q-1}\otimes \mathbb{V} \arrow{r}{\curl} &H^{q-2}\otimes \mathbb{V} \arrow{r}{\div} & H^{q-3}\otimes \mathbb{R} \arrow{r}{} & 0\\
0 \arrow{r}&H^{q-1}\otimes \mathbb{V}\arrow{r}{\grad} \arrow[ur, "\mathrm{id}"]&H^{q-2}\otimes \mathbb{M}  \arrow{r}{\curl} \arrow[ur, "2\vskw"]&H^{q-3}\otimes \mathbb{M} \arrow{r}{\div}\arrow[ur, "\tr"] & H^{q-4}\otimes \mathbb{V} \arrow{r}{} & 0\\
0 \arrow{r} &H^{q-2}\otimes \mathbb{V}\arrow{r}{\grad} \arrow[ur, "-\mskw"]&H^{q-3}\otimes \mathbb{M}  \arrow{r}{\curl} \arrow[ur, "S"]&H^{q-4}\otimes \mathbb{M} \arrow{r}{\div}\arrow[ur, "2\vskw"] & H^{q-5}\otimes \mathbb{V} \arrow{r}{} & 0\\
0 \arrow{r} &H^{q-3}\otimes \mathbb{R}\arrow{r}{\grad} \arrow[ur, "\iota"]&H^{q-4}\otimes \mathbb{V}  \arrow{r}{\curl} \arrow[ur, "-\mskw"]&H^{q-5}\otimes \mathbb{V} \arrow{r}{\div}\arrow[ur, "\mathrm{id}"] & H^{q-6}\otimes \mathbb{R} \arrow{r}{} & 0.
 \end{tikzcd}
\end{equation}
The first and the last rows of this diagram are simply the usual Sobolev de~Rham complex with two different Sobolev orders presented using vector proxies, while the middle two rows are each a Sobolev de~Rham complex tensored with the 3-dimensional space $\V$.  (Note that, because we tensor the de~Rham complex on the right, the differential operators in the middle two rows are applied columnwise: e.g., $\grad$ applied to a vector field is the matrix field whose columns are the gradients of the components of the field.)  It is elementary to check the anticommutativity around any of the six small parallelograms in the diagram.  Finally, for the maps connecting the first two rows, the $J$-injectivity/surjectivity conditions hold for $J=0$.  For the next two, it holds for $J=1$, and for the last two, for $J=2$.  We have thus verified all
the requirements to derive a new complex from any of the three pairs of consecutive rows.

From the first two rows of \eqref{diagram-4rows} we obtain in this way
the {\it Hessian complex} 
\begin{equation}\label{grad-grad0}
\begin{tikzcd}
0\arrow{r}  & H^{q}\otimes \mathbb{R} \arrow{r}{\gg} & H^{q-2}\otimes \mathbb{S} \arrow{r}{\curl} & H^{q-3}\otimes \mathbb{T} \arrow{r}{\div} & H^{q-4}\otimes \mathbb{V} \arrow{r} & 0,
\end{tikzcd}
\end{equation}
where $\gg:=\grad\grad$.
 From the second and third rows of \eqref{diagram-4rows} we obtain the {\it elasticity complex}
 \begin{equation}\label{sequence:hs}
\begin{tikzcd}
0\arrow{r} & H^{q-1}\otimes \mathbb{V} \arrow{r}{{\deff}} & H^{q-2}\otimes \mathbb{S} \arrow{r}{\inc} & H^{q-4}\otimes \mathbb{S} \arrow{r}{\div} & H^{q-5}\otimes \mathbb{V} \arrow{r} & 0.
\end{tikzcd}
\end{equation}
In this sequence, the middle operator $\inc = \curl S^{-1}\curl$ is a second order differential operator mapping matrix fields to matrix fields. Now, the curl of a symmetric matrix is trace-free (this follows from the anticommutativity of the second parallelogram of \eqref{diagram-4rows}).
It follows that $S^{-1}\curl u= {\mathrm T}\curl u$ (with $\mathrm T$ being the transpose operator) and so $\inc u = \curl {\mathrm T}\curl u$ for $u$ symmetric.  Thus
the definition of $\inc$ here extends that given after \eqref{sequence:3Delasticity}.
It is also easy to compute the action of $\inc$ on skew symmetric matrix fields. If $u$ is skew symmetric then $S^{-1}\curl u$ is the gradient of a vector field, as follows from the anticommutativity of the third parallelogram in \eqref{diagram-4rows}.  Thus $\inc u = \curl S^{-1}\curl u$ vanishes for skew $u$.

Finally we consider the last two rows of \eqref{diagram-4rows}. From these we derive the \emph{$\div\div$ complex}
 \begin{equation}\label{div-div0}
\begin{tikzcd}[row sep=large, column sep = large]
0 \to H^{q-2}\otimes \mathbb{V}  \arrow{r}{\dev\grad} & H^{q-3}\otimes \mathbb{T}  \arrow{r}{\sym\curl} & H^{q-4}\otimes \mathbb{S}  \arrow{r}{\div\div} & H^{q-6}\otimes \mathbb{V} \to 0.
\end{tikzcd}
\end{equation}
In addition to many applications of the elasticity complex as mentioned in the introduction, Pauly and Zulehner \cite{pauly2016closed} investigated the Hessian complex \eqref{grad-grad0} and the $\div\div$ complex \eqref{div-div0} with $H(\mathscr{D})$ type spaces. 
  See also \cite{pauly2018divdiv} for applications to the biharmonic equation. 

These examples, limited to three dimensions and using vector proxies instead of differential forms, are elementary but somewhat ad hoc. In the following subsection we generalize this example to $n$ dimensions using the language of differential forms, which renders it more systematic.


\subsection{Complexes from \texorpdfstring{$\alt^{k}$-valued}{Alt-k-valued} forms}\label{sec:first-example}

In this section we work in $n$ dimensions, so $\Omega$ is a domain in $\R^n$.
For $i\ge 0$, let $\alt^{i}=\alt^i \R^n$ be the space of algebraic $i$-forms, that is, of alternating $i$-linear maps on $\mathbb{R}^{n}$.  We also set $\alt^{i, J}=\alt^{i}\otimes \alt^{J}$, the space of $\alt^{J}$-valued $i$-forms or, equivalently, the space of $(i+J)$-linear maps on $\mathbb{R}^{n}$ which are alternating in the first $i$ variables and also in the last $J$ variables. Thus $\dim \alt^{i, J}=\binom{n}{i}\binom{n}{J}$.
For the linking maps, we define the algebraic operators $s^{i, J}: \alt^{i, J}\to \alt^{i+1, J-1}$
\begin{multline*}
s^{i, J}\mu (v_{0},\cdots,  v_{i})(w_{1}, \cdots, w_{J-1}):=\sum_{l=0}^{i}(-1)^{l} \mu (v_{0},\cdots, \widehat{v_{l}}, \cdots,  v_{i})(v^{l}, w_{1}, \cdots, w_{J-1}), \\ \forall v_{0}, \cdots, v_{i}, w_{1}, \cdots, w_{J-1} \in \mathbb{R}^{n}.
\end{multline*}
We also write $S^{i,J}=\mathrm{id}\otimes s^{i,J}:H^q\otimes\alt^{i, J}\to H^q\otimes\alt^{i+1, J-1}$ for any Sobolev order $q$. 
Now $H^{q}\otimes \alt^{i}$ is just another notation for $H^{q}\Lambda^{i}$, and so we have the exterior derivative, $d^{i}: H^{q}\otimes \alt^{i}\to H^{q-1}\otimes \alt^{i+1}$. Tensoring with $\alt^{J}$ then gives $d^{i}: H^{q}\otimes \alt^{i, J}\to H^{q-1}\otimes \alt^{i+1, J}$, where we have simply written $d^{i}$ in favor of $d^{i}\otimes \mathrm{id}_{\alt^{J}}$.  With these definitions, we may write down the diagram generalizing \eqref{diagram-4rows} to
$n$ dimensions:
\begin{equation}\label{diagram-nD}
\begin{tikzcd}[column sep=tiny]
0 \arrow{r} & H^{q}\otimes\alt^{0,0}  \arrow{r}{d} &H^{q-1}\otimes\alt^{1,0}  \arrow{r}{d} & \cdots \arrow{r}{d} & H^{q-n}\otimes \alt^{n,0} \arrow{r}{} & 0\\
0 \arrow{r} & H^{q-1}\otimes\alt^{0,1}  \arrow{r}{d} \arrow[ur, "S^{0,1}"] &H^{q-2}\otimes\alt^{1,1}  \arrow{r}{d}  \arrow[ur, "S^{1,1}"] & \cdots \arrow{r}{d}  \arrow[ur, "S^{n-1,1}"] & H^{q-n-1}\otimes \alt^{n,1} \arrow{r}{} & 0\\[-15pt]
 & \vdots & \vdots & {} & \vdots & {} \\[-15pt]
0 \arrow{r} & H^{q-n+1}\otimes\alt^{0,n-1}  \arrow{r}{d} &H^{q-n}\otimes\alt^{1,n-1}  \arrow{r}{d} & \cdots \arrow{r}{d} & H^{q-2n+1}\otimes \alt^{n,n-1} \arrow{r}{} & 0\\
0 \arrow{r} & H^{q-n}\otimes\alt^{0,n}  \arrow{r}{d} \arrow[ur, "S^{0,n}"] &H^{q-n-1}\otimes\alt^{1,n}  \arrow{r}{d}  \arrow[ur, "S^{1,n}"] & \cdots \arrow{r}{d}  \arrow[ur, "S^{n-1,n}"] & H^{q-2n}\otimes \alt^{n,n} \arrow{r}{} & 0.
\end{tikzcd}
\end{equation}

As before, we can take any pair of consecutive rows and apply the general algebraic construction of
Section~\ref{sec:framework}.  Specifically, we fix an arbitrary real number $q$ and an integer $J$ with $0\leq J < n$, and let $V^{i}:=H^{q-J-i}$, $\mathbb{E}^{i}:=\alt^{i, J}$, and $\tilde{\mathbb{E}}^{i}:=\alt^{i, J+1}$. The differentials $D^i$ and $\tilde D^i$ are then just the exterior derivatives
$d:H^{q-J-i}\otimes \alt^{i, J}\to H^{q-J-i-1}\otimes \alt^{i+1, J}$ and $d: H^{q-J-i-1}\otimes \alt^{i, J+1}\to H^{q-J-i-2}\otimes \alt^{i+1, J+1}$. In short, the top complex in \eqref{Z-complex} is the Sobolev de~Rham complex of order $q-J$ tensored with $\alt^J$ and the bottom complex is the Sobolev de~Rham complex of order $q-J-1$ tensored with $\alt^{J+1}$.  Finally, the linking map $s^i$ from $\tilde{\mathbb{E}}^i\to \mathbb{E}^{i+1}$, i.e., from
$\alt^{i, J+1}$ to $\alt^{i+1,J}$, is the just the natural map $s^{i,J+1}$ obtained by skew-symmetrization.

We now verify the requirements on the linking maps.
\begin{lemma}\label{SDDS-example}
With the differentials $D^i$, $\tilde D^i$ and the linking maps $s^i=s^{i,J+1}$ defined as above,
the anticommutativity condition \eqref{SDDS} holds.
\end{lemma}
We will prove this result shortly, as a corollary of Lemma~\ref{lem:dKKd0}.
\begin{lemma}\label{lem:s-insurj}
The operators $s^i = s^{i, J+1}$ are injective for $0\le i\leq J$ and surjective for $J\le i\le n$.
\end{lemma}
We prove Lemma \ref{lem:s-insurj} in Appendix 1.

In this example, the output complex \eqref{reduced-complex} reads
\begin{multline}\label{seq:alt}
\cdots \to H^{q-2J+1}\otimes \ran (s^{J-2,J+1})^{\perp} \to H^{q-2J}\otimes \ran (s^{J-1,J+1})^{\perp} \xrightarrow{d\circ (S^{J,J+1})^{-1}\circ d} 
\\
H^{q-2J-2}\otimes \ker(s^{J+1,J+1})  \to H^{q-2J-3}\otimes \ker(s^{J+2,J+1}) \to  \cdots 
\end{multline}

In this way we have derived $n$ new complexes, one for each choice of $J$ with $0\le J<n$.  Each involves $n+1$ spaces and $n$ differential operators, with all the operators of first order except for one.  It follows from Theorem~\ref{thm:dimension} that each complex has finite dimensional cohomology and thus the differentials have closed range. In the case $n=3$, these complexes are the Hessian complex, the elasticity complex, and the div-div complex previously derived.

\subsubsection*{Explicit representation of cohomology}
In Theorem \ref{thm:dimension}, \eqref{SNR} provides a condition for obtaining the exact dimension and specific representations of the cohomology. Next, we will introduce a Koszul type operator as required in \eqref{SDKKD} to verify this condition for the above example.  

The first step to construct such operators is to introduce the Koszul operator $\tilde{K}^{J}: H^{q}\otimes \alt^{J}\to H^{q}\otimes \alt^{J-1}$, with any real number $q$, defined by a contraction with the Euler (identity) vector field, i.e., 
\begin{align}\label{Ktilde}
\tilde{K}^{J}\mu (w_{1}, \cdots, w_{J-1}):=\mu (x, w_{1}, \cdots, w_{J-1}), \quad \forall w_{1}, \cdots, w_{J-1} \in \mathbb{R}^{n},
\end{align}
where 
$x$ is the Euler (identity) vector field in $\mathbb{R}^{n}$.  In terms of the standard coordinates on $\R^n$, we have
$$
\tilde K^J(f\,dx^{\tau_{1}}\wedge \cdots \wedge dx^{\tau_{J}})
 = \sum_{j=1}^J(-1)^{j-1}x^{\tau_j} fd\,x^{\tau_{1}}\wedge\cdots\widehat{dx^{\tau_j}} \cdots \wedge \cdots \wedge dx^{\tau_{J}}.
 $$
 where $f=f(x)$ is an arbitrary coefficient function and $\widehat{dx^{\tau_j}}$ indicates that that factor is omitted from the wedge product.
Tensoring with $\alt^{i}$, we extend the above Koszul operator to $\tilde{K}^{i,J}: H^{q}\otimes \alt^{i, J}\to H^{q}\otimes \alt^{i, J-1}$. 
\begin{lemma}\label{lem:dKKd0}
We have 
\begin{align}\label{dKKd0}
S^{i, J}=d^{i}\tilde{K}^{i,J}-\tilde{K}^{i+1,J}d^{i}.
\end{align}
\end{lemma}
\begin{proof}
We may expand an arbitrary element of $H^q\otimes \alt^{i, J}$ as a sum of terms of the form
$$
\mu:= f\, dx^{\sigma_{1}}\wedge \cdots \wedge dx^{\sigma_{i}}\otimes dx^{\tau_{1}}\wedge \cdots \wedge dx^{\tau_{J}},
$$
where $1\le\sigma_1<\cdots<\sigma_i\le n$, $1\le\tau_1<\cdots<\tau_J\le n$, and $f\in H^q(\Omega)$.
Thus it suffices to prove that $S^{i, J}\mu=d^{i}\tilde{K}^{i,J}\mu-\tilde{K}^{i+1,J}d^{i}\mu$ for such $\mu$.
Now
$$
\tilde K^{i,j}\mu=\sum_{j=1}^J(-1)^{j-1}x^{\tau_j}f\,dx^{\sigma_{1}}\wedge \cdots \wedge dx^{\sigma_{i}}\otimes 
dx^{\tau_{1}}\wedge\cdots\widehat{dx^{\tau_j}} \cdots \wedge \cdots \wedge dx^{\tau_{J}}.
$$
and
\begin{align*}
d^{i}&\tilde K^{i,J}\mu
\\
&=\sum_{l=1}^n\sum_{j=1}^J(-1)^{j-1}\frac{\partial(x^{\tau_j}f)}{\partial x^l}dx^l\wedge dx^{\sigma_{1}}\wedge \cdots \wedge dx^{\sigma_{i}}\otimes 
dx^{\tau_{1}}\wedge\cdots\widehat{dx^{\tau_j}} \cdots \wedge \cdots \wedge dx^{\tau_{J}}
\\
&=\sum_{j=1}^J(-1)^{j-1}x^{\tau_j}\sum_{l=1}^n\frac{\partial f}{\partial x^l}dx^l\wedge dx^{\sigma_{1}}\wedge \cdots \wedge dx^{\sigma_{i}}\otimes 
dx^{\tau_{1}}\wedge\cdots\widehat{dx^{\tau_j}} \cdots \wedge \cdots \wedge dx^{\tau_{J}}
\\
&\quad + \sum_{j=1}^J(-1)^{j-1}f\,dx^{\tau_j}\wedge dx^{\sigma_{1}}\wedge \cdots \wedge dx^{\sigma_{i}}\otimes 
dx^{\tau_{1}}\wedge\cdots\widehat{dx^{\tau_j}} \cdots \wedge \cdots \wedge dx^{\tau_{J}}
\\
&=\tilde K^{i+1,J} d^i \mu + S^{i,J}\mu.
\end{align*}
\end{proof}
An immediate consequence of the lemma is the identity
\begin{equation}\label{dSSd}
d^{i+1}S^{i,J}=-S^{i+1,J}d^i,
\end{equation}
which establishes Lemma~\ref{SDDS-example}.

Lemma~\ref{lem:dKKd0} suggests that we use the operators $\tilde K^{i,J}$ to obtain the condition \eqref{SDKKD}.  However, these operators do not satisfy the necessary boundedness.  The condition requires an operator $K^i$
mapping $H^{q-J-i-1}\otimes \alt^{i, J+1}$ boundedly into $H^{q-J-i}\otimes \alt^{i, J}$, i.e., which smooths by one order of differentiability, but the operators $\tilde K^{i,J}$ are not smoothing.  To address this, we make use of homotopy operators for the Sobolev de~Rham complex established by Costabel and McIntosh  \cite{costabel2010bogovskiui}. 
\begin{lemma}
For the complex \eqref{sobolev-deRham}, there exist $P^{i}: H^{q-i}\Lambda^{i}\rightarrow H^{q-i+1}\Lambda^{i-1}$ and $L^{i}: H^{q-i}\Lambda^{i} \rightarrow C^{\infty}\Lambda^{i}$ with finite dimensional range, for $i=1, 2, \cdots, n$, satisfying
\begin{align}\label{dPPdL}
d^{i-1}P^{i}+P^{i+1}d^{i}=\mathrm{id}-L^{i}, \quad i=1, 2, \cdots, n.
\end{align} 
\end{lemma}
Note that, from \eqref{dPPdL}, we have the commutativity
\begin{equation}\label{comm}
d^{i}L^{i}=L^{i+1}d^{i}.
\end{equation}
Now we define the operator $K^i$ in \eqref{SDKKD} by
\begin{equation}\label{defK}
K^i=P^{i+1}S^{i, J}+L^{i}\tilde{K}^{i,J}.
\end{equation}
Then $K^i$ maps $H^{q}\otimes \alt^{i, J}$ boundedly into $H^{q+1}\otimes \alt^{i, J-1}$ for any real number $q$.  Moreover, condition \eqref{SDKKD} is still fulfilled.
\begin{lemma}
Let $0\le i\le n$ and $0\le J <n$ be integers and let  $q$ be any real number. Then
\begin{align}\label{sobolev-K}
(d^{i}K^{i,J}-K^{i+1,J}d^{i})\mu=S^{i, J}\mu,\quad\mu\in H^{q}\otimes \alt^{i, J}.
\end{align}
\end{lemma}
\begin{proof}
Using \eqref{defK}, \eqref{dSSd}, \eqref{dPPdL} , \eqref{comm}, and \eqref{dKKd0}, we obtain
\begin{multline*}
dK-Kd=d(PS+L\tilde{K})-(PS+L\tilde{K})d=dPS+PdS+dL\tilde{K}-L\tilde{K}d \\ =(\mathrm{id}-L)S+Ld\tilde{K}-L\tilde{K}d=S.
\end{multline*}
\end{proof}
Having verified condition \eqref{SDKKD}, we obtain \eqref{SNR} thanks to Proposition~\ref{prop:K}.
Therefore we may apply Theorem~\ref{thm:dimension} to conclude that the dimension of the $i$th
cohomology space for the output complex is precisely the sum of the corresponding dimensions for the two input de~Rham complexes.  Moreover, if we choose explicit spaces of cohomology representatives for the input  de~Rham complexes, we may apply Corollary~\ref{cor:rep} to obtain the explicit representation \eqref{explicit-cohomology} of the cohomology of the output complex.  If the cohomology representatives for the input complexes are chosen to be independent of the Sobolev index $q$, as in Theorem~\ref{thm:CM}, then the resulting representatives for the output complex will have the same property. From this follows a variety of properties for the output complex as discussed in Section~\ref{sec:preliminary} (existence of regular potentials, regular decomposition, compactness property, etc.).

\subsection{More complexes from \texorpdfstring{$\alt^k$-valued}{Alt-k-valued} forms}\label{sec:iterated}
In the previous two sections we took as the input complexes two consecutive rows of the diagram \eqref{diagram-4rows}
(in three dimensions) or its $n$-dimensional generalization \eqref{diagram-nD}.  Actually, it is not necessary that the rows be consecutive.  To illustrate, we derive a new complex taking as input complexes
the first and third rows in \eqref{diagram-4rows}.  For the connecting operators, we compose two
$S$ operators, multiply the first composition by $-1$ to retain the anticommutativity, and divide each by $2$ for convenience.
Noting that $\mskw\circ\vskw$ is the identity on $\V$ and $\tr\circ S = -2\tr$, we are led to following diagram in which we have added some additional zeros to line up the two complexes:
\begin{equation}\label{Z-complex-1-3}
\begin{tikzcd}[column sep=1.6em]
0 \arrow{r} &H^{q}\arrow{r}{\grad} &H^{q-1}\otimes \mathbb{V} \arrow{r}{\curl}&H^{q-2}\otimes \mathbb{V}  \arrow{r}{\div} & H^{q-3} \arrow{r}{} & 0 \arrow{r}{} & 0\\
0 \arrow{r}&0 \arrow{r} \arrow[ur, "0"] &H^{q-2}\otimes \mathbb{V}\arrow{r}{\grad} \arrow[ur, "\mathrm{id}"]& H^{q-3} \otimes \mathbb{M} \arrow{r}{\curl} \arrow[ur, "-\tr"]&H^{q-4} \otimes \mathbb{M} \arrow{r}{\div}\arrow[ur, "0"] &H^{q-5} \otimes \mathbb{V} \arrow{r}{} & 0.
 \end{tikzcd}
\end{equation}
One may easily verify that \eqref{Z-complex-1-3} satisfies the assumptions of Section~\ref{sec:framework}, and so we derive a new complex from it.
In this case, the output complex \eqref{reduced-complex} turns out to be the {\it $\grad\curl$ complex:}
\begin{equation}\label{grad-curl}
\begin{tikzcd}[column sep=1.5em]
0 \arrow{r}{}&H^{q}\arrow{r}{\grad} & H^{q-1}\otimes \mathbb{V}  \arrow{r}{\grad\curl} &[12] H^{q-3}\otimes \mathbb{T} \arrow{r}{\curl} & H^{q-4}\otimes \mathbb{M} \arrow{r}{\div} & H^{q-5}\otimes \mathbb{V} \arrow{r} & 0.
\end{tikzcd}
\end{equation}
The second order operator appearing in this complex, $\grad\curl$, appears in several applications. In Cosserat elasticity and couple stress models,  it is introduced to incorporate the size effects, c.f., \cite{mindlin1962effects,park2008variational}. 
We also refer to \cite{chacon2007steady}  for a $\grad\curl$ correction term  in magnetohydrodynamics problems.

In a similar way, we may take the second and fourth rows in \eqref{diagram-4rows} as inputs and derive the \emph{$\curl$ $\div$ complex:}
\begin{equation}\label{curl-div}
\begin{tikzcd}[column sep=1.4em]
0 \to H^{q}\otimes \mathbb{V}\arrow{r}{{\grad}} & H^{q-1}\otimes \mathbb{M}  \arrow{r}{\dev\curl} &[12] H^{q-2}\otimes \mathbb{T} \arrow{r}{\curl\div} &[12] H^{q-4}\otimes \mathbb{V} \arrow{r}{\div} & H^{q-5} \to 0.
\end{tikzcd}
\end{equation}
The $\curl\div$ operator for trace-free matrix fields appears in several applications including couple stress models and Cosserat elasticity, see, for example, \cite[equation 1.16]{mindlin1962effects}. The deviator of the couple-stress is a trace-free matrix field.
Gopalakrishnan, Lederer and Sch\"{o}berl \cite{gopalakrishnan2018mass} proposed a mass conserving mixed stress formulation for the Stokes problems where the $\curl \div$ operator plays a role.

We may even take the first and the last rows of  \eqref{diagram-4rows} as the input complexes.  Then there is only one nonzero linking map, obtained by composing three of the $s^i$ operators.  After multiplication by a constant it is just the identity from the first space in the last row to the last space in the first row.
This leads to the \emph{$\grad\div$ complex:}
\begin{equation}\label{grad-div}
\begin{tikzcd}[row sep=small]
0 \to H^{q}\arrow{r}{\grad} & H^{q-1} \otimes \mathbb{V}  \arrow{r}{\curl} & H^{q-2}\otimes \mathbb{V} 
\\
 & \arrow{r}{\grad\div} &[12] H^{q-4}\otimes \mathbb{V} \arrow{r}{\curl} & H^{q-5}\otimes \mathbb{V} \arrow{r}{\div} &  H^{q-6} \to 0.
\end{tikzcd}
\end{equation}

Applying Theorem \ref{thm:dimension}, we conclude that the cohomology of each of the above complexes, i.e., \eqref{grad-curl}, \eqref{curl-div} and \eqref{grad-div}, has finite dimension. We could also define $K$ operators satisfying \eqref{SDKKD} as was done in Section~\ref{sec:first-example}, verifying the conditions in Theorem~\ref{thm:rep} and thus giving an expression for the dimension of the cohomology spaces, and explicit representation of the cohomology in terms of representations of de~Rham cohomology.

\subsection{Iterating the construction}\label{sec:conformal}

In the preceding section we derived various complexes starting from two de~Rham complexes.  Next we use two of those output complexes as input to the construction, and thereby derive a new complex.  This complex includes a space of matrix fields which are both symmetric and trace-free, a class of fields which has numerous applications. Specifically, consider the following diagram whose three rows are the Hessian, elasticity, and $\div\div$ complexes derived above:
\begin{equation}\label{diagram-1}
\begin{tikzcd}
0 \arrow{r}{}&H^{q}\otimes \mathbb{V}\arrow{r}{\dev\grad} & H^{q-1}\otimes \mathbb{T}  \arrow{r}{\sym\curl} & H^{q-2}\otimes \mathbb{S} \arrow{r}{\div\div} & H^{q-4} \arrow{r}{} &  0\\
0 \arrow{r}{}&H^{q-1}\otimes \mathbb{V}\arrow{r}{\deff}\arrow[ur, "-\mskw"]  & H^{q-2}\otimes \mathbb{S}  \arrow{r}{\inc}\arrow[ur, "S"]  & H^{q-4}\otimes \mathbb{S} \arrow{r}{\div}\arrow[ur, "\tr"]  & H^{q-5}\otimes \mathbb{V} \arrow{r}{} &  0\\
0 \arrow{r}{}&H^{q-2}\arrow{r}{\gg}\arrow[ur, "\iota"] & H^{q-4}\otimes \mathbb{S}  \arrow{r}{\curl} \arrow[ur, "S"]& H^{q-5}\otimes \mathbb{T} \arrow{r}{\div}\arrow[ur, "2\vskw"] & H^{q-6}\otimes \mathbb{V} \arrow{r}{} &  0.
\end{tikzcd}
\end{equation}
Either by direct calculation or by the commutativity of \eqref{diagram-4rows}, it is elementary to check that this diagram anticommutes and satisfies the injectivity/surjectivity condition (with the operator $S$ being the bijective linking map in both rows), so we may apply the algebraic construction to either the first and second rows or to the second and third rows.  We obtain the same output complex in both cases, namely
\begin{equation}\label{TTcomplex}
\begin{tikzcd}[column sep=2.5em]
0 \to H^{q}\otimes \mathbb{V}\arrow{r}{\dev\deff} & H^{q-1}\otimes  (\mathbb{S}\cap \mathbb{T}) \arrow{r}{\cinc} & H^{q-4}\otimes (\mathbb{S}\cap \mathbb{T}) \arrow{r}{\div} & H^{q-5}\otimes \mathbb{V}  \to  0.
\end{tikzcd}
\end{equation}
Here the third order differential operator
$$
\cinc:=\curl S^{-1}\inc = \curl S^{-1}\curl S^{-1}\curl=\inc S^{-1}\curl.
$$
Note that, if $v$ is skew, then $S^{-1}v =-v$ is also skew, so $\inc S^{-1}v=0$. Thus, for any matrix field $u$, $\inc S^{-1}\sym\curl u=\inc S^{-1}\curl u= \cinc u$,
so the operator derived from the first two rows of \eqref{diagram-1} is indeed $\cinc$.

We refer to \cite{beig1996tt} and the references therein for a smooth version of \eqref{TTcomplex} and its applications in general relativity. 
The complex \eqref{TTcomplex}, which may be referred to as the conformal elasticity complex or just the conformal
complex, is in many ways analogous to the elasticity complex \eqref{sequence:hs}.  Like \eqref{sequence:hs},  \eqref{TTcomplex} is formally self-adjoint.  The operator $\cinc$ plays the role of  $\inc$ in the elasticity complex.  While the elasticity complex is locally a resolution of the 6-dimensional space of infinitesimal rigid motions (Killing fields), the complex \eqref{TTcomplex} is locally a resolution of the 10-dimensional space of conformal Killing fields, i.e., fields $v$ for which $\dev\deff v$ vanishes.  From the elasticity complex we obtain Korn's inequality as one of the Poincar\'e inequalities of the complex, bounding the $H^1$ norm of a vector field by the $L^2$ norm of its deformation as long as the field is orthogonal to the Killing fields.  In the same way, from the complex \eqref{TTcomplex}, we obtain the stronger trace-free Korn's inequality which bounds the $H^1$ norm by the $L^2$ norm of the trace-free part of its deformation, as long as the field is orthogonal to the conformal Killing fields.

The spaces and operators appearing in \eqref{TTcomplex} have numerous applications in general relativity and continuum mechanics. For example, Dain \cite{dain2006generalized} used them to study the momentum constraints in the Cauchy problem for the Einstein equations while Fuchs and Schirra \cite{fuchs2009application} investigated applications in relativity and Cosserat elasticity. Further, the recently proposed mass conserving mixed formulation of the Stokes equations by
Gopalakrishnan, Lederer and Sch\"{o}berl \cite{gopalakrishnan2019mass} is related to the last several spaces in \eqref{TTcomplex}.
Similarly, the trace-free Korn's inequality has various applications, e.g., to
fluid dynamics 
\cite[Proposition 2.1]{feireisl2009singular} and to Cosserat elasticity \cite{breit2017trace,jeong2009numerical,neff2009new}.  See \cite{breit2017trace} for more references on this.

Another complex can be derived if we start with the Hessian complex and the de~Rham complex with appropriate linking maps:
\begin{equation}
\begin{tikzcd}
0 \arrow{r} &H^{q}\otimes \mathbb{R}\arrow{r}{\gg} &H^{q-2}\otimes \mathbb{S} \arrow{r}{\curl} &H^{q-3}\otimes \mathbb{T}\arrow{r}{\div}&H^{q-4}\otimes \mathbb{R}\arrow{r}{}&0\\
0 \arrow{r} &H^{q-2}\otimes \mathbb{R} \arrow{r}{\grad} \arrow[ur, "\iota"] &H^{q-3}\otimes \mathbb{V}  \arrow{r}{\curl} \arrow[ur, "-\mskw"] &H^{q-4}\otimes \mathbb{V}  \arrow{r}{\div} \arrow[ur, "\mathrm{id}"] &H^{q-5}\otimes \mathbb{R}\arrow{r}{}&0.
 \end{tikzcd}
\end{equation}
The output complex which results is the {\it conformal Hessian complex}
\begin{equation}\label{conformal-hess}
\begin{tikzcd}[column sep=2.5em]
0 \to H^{q}\otimes \mathbb{R} \arrow{r}{\dev\gg} &H^{q-2}\otimes (\mathbb{S}\cap \mathbb{T})  \arrow{r}{\sym\curl} &H^{q-3}\otimes (\mathbb{S}\cap \mathbb{T})  \arrow{r}{\div\div} &H^{q-5}\otimes \mathbb{R} \to 0.
 \end{tikzcd}
\end{equation}
Similarly, we can start with the de~Rham complex and the $\div\div$ complex:
\begin{equation}
\begin{tikzcd}
0 \arrow{r} &H^{q}\otimes \mathbb{R} \arrow{r}{\grad} &H^{q-1}\otimes \mathbb{V}  \arrow{r}{\curl} &H^{q-2}\otimes \mathbb{V}  \arrow{r}{\div} &H^{q-3}\otimes \mathbb{R}\arrow{r}{}&0 \\
0 \arrow{r} &H^{q-1}\otimes \mathbb{V}\arrow{r}{\dev\grad} \arrow[ur, "\mathrm{id}"]&H^{q-2}\otimes \mathbb{T} \arrow{r}{\sym\curl} \arrow[ur, "2\vskw"]&H^{q-3}\otimes \mathbb{S}\arrow{r}{\div\div} \arrow[ur, "\tr"] &H^{q-5}\otimes \mathbb{R}\arrow{r}{}&0.
 \end{tikzcd}
\end{equation}
The output complex is again the conformal Hessian complex \eqref{conformal-hess}.
Because of its relation with the Hamiltonian constraint in general relativity, it is referred to as
the \emph{Hamiltonian complex} in \cite{beig2020linearised}.  Using similar techniques we can derive
the \emph{momentum complex} in \cite{beig2020linearised}, so named because of its relation to the momentum
constraint of relativity.

\subsection{Two space dimensions}

Most of the examples presented above in 3D have analogues in 2D. In this section we briefly summarize the output complexes in 2D. 
First we introduce some notation.
In $\mathbb{R}^{2}$, a skew symmetric matrix can be identified with a scalar. Using the same notation as in 3D, we let $\mskw: \mathbb{R}\to \mathbb{K}$ be this identification, i.e.,  
$$
\mskw(u):= \left ( 
\begin{array}{cc}
0 &u\\
-u & 0
\end{array}
\right )\quad \mbox{in } \mathbb{R}^{2}.
$$
We also let $\sskw=\mskw^{-1}\circ \skw: \mathbb{M}\to \mathbb{R}$ be the map taking the skew part of a matrix and identifying it with a scalar.

The 2D analogue of the diagram \eqref{diagram-4rows} is
\begin{equation}\label{diagram-3rows2D}
\begin{tikzcd}
0 \arrow{r}  &H^{q}\otimes \mathbb{R} \arrow{r}{\grad} &H^{q-1}\otimes \mathbb{V} \arrow{r}{\rot} & H^{q-2}\otimes \mathbb{R} \arrow{r}{} & 0\\
0 \arrow{r} & H^{q-1}\otimes \mathbb{V} \arrow{r}{\grad}\arrow[ur, "\mathrm{id}"] &H^{q-2}\otimes \mathbb{M} \arrow{r}{\rot}\arrow[ur, "-2\sskw"] & H^{q-3}\otimes \mathbb{V} \arrow{r}{} & 0\\
0 \arrow{r}  &H^{q-2}\otimes \mathbb{R} \arrow{r}{\grad}\arrow[ur, "\mskw"] &H^{q-3}\otimes \mathbb{V} \arrow{r}{\rot}\arrow[ur, "\mathrm{id}"] & H^{q-4}\otimes \mathbb{R} \arrow{r}{} & 0.
 \end{tikzcd} 
\end{equation}
The output complexes using two consecutive rows read:
\begin{equation*}
\begin{tikzcd}
0 \arrow{r}& H^{q} \arrow{r}{\gg} & H^{q-2}\otimes \mathbb{S}  \arrow{r}{\rot} & H^{q-3}\otimes \mathbb{V} \arrow{r} & 0,
\end{tikzcd}
\end{equation*}
and
\begin{equation*}
\begin{tikzcd}[column sep=3em]
0  \arrow{r}& H^{q-1}\otimes \mathbb{V} \arrow{r}{\deff} & H^{q-2}\otimes \mathbb{S}  \arrow{r}{\rot\rot} & H^{q-4} \arrow{r} & 0,
\end{tikzcd}
\end{equation*}
respectively. 
Using the first and last rows, we obtain the following diagram:
\begin{equation}\label{BGG:degree-2-2}
\begin{tikzcd}
0 \arrow{r} &H^{q}\arrow{r}{\grad} &H^{q-1}\otimes \mathbb{V} \arrow{r}{\rot}&H^{q-2}\arrow{r}{} & 0 \arrow{r}{} & 0\\
0 \arrow{r}& 0 \arrow{r} \arrow[ur, "0"] &H^{q-2}\arrow{r}{\grad} \arrow[ur, "\mathrm{id}"]& H^{q-3} \otimes \mathbb{V} \arrow{r}{\rot}\arrow[ur, "0"] &H^{q-4} \arrow{r}{}\arrow[ur, "0"] & 0.
 \end{tikzcd}
\end{equation}
 This leads to the output complex
\begin{equation}
\begin{tikzcd}[column sep=3em]
0 \arrow{r} & H^{q} \arrow{r}{\grad} & H^{q-1}\otimes \mathbb{V} \arrow{r}{\grad\rot} & H^{q-3}\otimes \mathbb{V} \arrow{r}{\rot} & H^{q-4}\arrow{r}& 0.
\end{tikzcd}
\end{equation}
On contractible domains, the cohomology at $H^{q-1}\otimes \mathbb{V}$ is $\mathbb{R}$. 

The conformal complexes \eqref{TTcomplex} and \eqref{conformal-hess} do \emph{not} immediately carry over to 2D.  That is because in the diagram
\begin{equation}\label{diagram-2D-reduction}
\begin{tikzcd}
0 \arrow{r}{} &H^{q}\otimes \mathbb{V} \arrow{r}{{\deff}} &H^{q-1}\otimes \mathbb{S} \arrow{r}{\rot\rot} & H^{q-3}\arrow{r} & 0\\
0 \arrow{r}{}&H^{q-1} \arrow{r}{{\gg}}\arrow[ur, "\iota"] &H^{q-3}\otimes \mathbb{S} \arrow{r}{\rot} \arrow[ur, "\tr"]& H^{q-4}\arrow{r} & 0,
 \end{tikzcd} 
\end{equation}
analogous to \eqref{diagram-1}, 
neither of the two linking maps is bijective.  The failure of this diagram to fulfil the requirements of our framework is consistent with the invalidity of the trace-free Korn's inequality in two dimensions.

\section{Proof of main results}\label{sec:proof}

In this section, we prove the main results on the dimension of cohomology and the cohomology isomorphism, i.e., Theorems~\ref{thm:dimension} and \ref{thm:rep}. To relate the cohomology of the input complexes, i.e., the $Z$ and the $\tilde{Z}$ complexes \eqref{Z-complex}, to the cohomology of the output complex \eqref{reduced-complex},  we follow two steps.  Throughout the section we assume that $Z^\bs$ and $\tilde Z^\bs$ are bounded Hilbert complexes and the $S^i$ are bounded linear operators satisfying \eqref{Z-complex}--\eqref{injsurj}. 

The first step, detailed in Section~\ref{step1}, is to construct a twisted direct sum of the $Z$ and the $\tilde{Z}$ complexes (this is \eqref{A-sequence-0}  below, which we refer to as the ``twisted complex''), and compare it with the direct sum of the $Z$ and the $\tilde{Z}$ complexes (referred to as the ``sum complex''). We will show that in general there exists a surjective map from the cohomology of the sum complex to the cohomology of the twisted complex. Therefore the dimension of the cohomology of the twisted complex is bounded by the sum of the dimensions of cohomology of the $Z$ and $\tilde{Z}$ complexes. Furthermore, the cohomology dimensions of the sum complex and of the twisted complex are equal if the condition $S\ker\subset \ran$ in Theorem \ref{thm:dimension} holds.

The second step, explained in Section~\ref{step2}, is to split the twisted complex into two subcomplexes. One of them is isomorphic to the output complex, while the other is exact on any domain, independent of its topology. Removing the exact sequence from the twisted complex does not change its cohomology. Thus we see that the cohomology of the output complex  is isomorphic to that of the twisted complex.

Combining the two steps we obtain the desired relation between the cohomology of the input complexes and of the output complex.

\subsection{From the sum complex to twisted complex}\label{step1}

The direct sum of the complexes $(Z^\bs, D^\bs)$ and $(\tilde{Z}^\bs,\tilde D^\bs)$ from \eqref{Z-complex} is the complex with the spaces
$Y^i:= Z^i\x \tilde Z^i$ and the differentials $D^i\x \tilde D^i$.  The \emph{twisted complex} has the same spaces, but the differentials are taken to be
\begin{equation}\label{def:A}
\mathscr{A}^{i}:=\left ( 
\begin{array}{cc}
D^{i} & -S^{i} \\
0 & \tilde{D}^{i}
\end{array}
\right ).
\end{equation}
Thus the twisted complex is
\begin{equation}\label{A-sequence-0}
\begin{tikzcd}
\cdots \arrow{r}{} &Y^{i-1} \arrow{r}{\mathscr{A}^{i-1}} &Y^{i} \arrow{r}{\mathscr{A}^{i}} & Y^{i+1}  \arrow{r}{} & \cdots,
 \end{tikzcd}
\end{equation}
which we write as $(Y^\bs,\mathscr{A}^\bs)$ or simply as $Y^\bs$ for short.
The anticommutativity \eqref{SDDS} implies the chain complex condition $\mathscr{A}^{i+1}\circ\mathscr{A}^i=0$.
In the remainder of this subsection we relate the cohomology of the twisted complex to that of the sum complex (or, equivalently, to that of the input complexes $Z^\bs$ and $\tilde Z^\bs$).

Let $H^{\bs}$ and $\tilde{H}^{\bs}$ be cohomology representatives for the $Z$ and $\tilde{Z}$ complexes, i.e., 
$$
\ker (D^{i}, Z^{i})=\ran(D^{i-1}, Z^{i-1})\oplus H^{i} \quad \text{and} \quad \ker (\tilde{D}^{i}, \tilde{Z}^{i})=\ran(\tilde{D}^{i-1}, \tilde{Z}^{i-1})\oplus \tilde{H}^{i}.
$$
Also, let $W^{\bs}$ complement $\ker(D^{\bs})$ in $Z^{\bs}$ and similarly for $\tilde{W}^{\bs}$. Thus
$$
Z^{i}=\ran(D^{i-1})\oplus H^{i} \oplus W^{i},\quad \mbox{and} \quad \tilde{Z}^{i}=\ran(\tilde{D}^{i-1})\oplus \tilde{H}^{i} \oplus \tilde{W}^{i}.
$$
Then $D^{i}: W^{i}\to \ran(D^{i})$ is an isomorphism whose inverse we denote $l^{i}: \ran(D^{i})\to W^{i}$.
\begin{lemma}\label{lem6}
\begin{align}\label{sum}
\ker(\mathscr{A}^{i})=\ran(\mathscr{A}^{i-1})+\left ( 
\begin{array}{cc}
I & l^{i+1}S^{i} \\
0 & I
\end{array}
\right )\left \{(h, \tilde{h}): h\in H^{i}, \tilde{h}\in \tilde{H}^{i}, S^{i}\tilde{h}\in \ran(D^{i}) \right \}.
\end{align}
\end{lemma}
\begin{proof}
For $(\omega, \mu)\in \ker(\mathscr{A}^{i})$, $D^{i}\omega-S^{i}\mu=0$ and $\tilde{D}^{i}\mu=0$. Therefore there exists $\alpha\in \tilde{Z}^{i-1}$ and $\tilde{h}\in \tilde{H}^{i}$ such that $\mu=\tilde{D}\alpha+\tilde{h}$.  Then $D\omega=S\tilde{D}\alpha+S\tilde{h}=-DS\alpha+S\tilde{h}$, and $D(\omega+S\alpha)=S\tilde{h}\in \ran{(D)}$.  Therefore $l$ is well defined on $S\tilde{h}$, and 
$$
D(\omega+S\alpha-lS\tilde{h})=0,
$$
which implies that 
$$
\omega+S\alpha-lS\tilde{h}=D\beta+h,
$$
for some $\beta\in Z^{i-1}$ and $h\in H^{i}$. Now we have verified that
$$
\left ( 
\begin{array}{c}
\omega\\
\mu
\end{array}
\right )=\left ( 
\begin{array}{cc}
D & -S\\
0 & \tilde{D}
\end{array}
\right )\left ( 
\begin{array}{c}
\beta\\
\alpha
\end{array}
\right )+\left ( 
\begin{array}{cc}
I & lS\\
0 &I
\end{array}
\right )\left ( 
\begin{array}{c}
h\\
\tilde{h}
\end{array}
\right ).
$$
This shows that the left-hand side of \eqref{sum} is contained in the right-hand side.
The opposite inclusion follows from the equation
$$
\left (
\begin{array}{cc}
D^{i} & -S^{i} \\
0 & \tilde D^{i}
\end{array}
\right )
\left (
\begin{array}{cc}
I & l^{i+1}S^{i} \\
0 & I
\end{array}
\right )\left (
\begin{array}{c}
h \\
\tilde{h}
\end{array}
\right )=0,
$$
which is easily verified.
\end{proof}
Under the assumption that $S$ induces the zero map on cohomology, we obtain an explicit set of cohomology representatives
for the twisted complex.
\begin{lemma}\label{lem7}
Assume that $S$ induces zero map on cohomology, i.e., $S^{i}\ker(\tilde{D}^{i})\subset \ran(D^{i})$. Then
$$
\ker(\mathscr{A}^{i})=\ran(\mathscr{A}^{i-1})\oplus\left ( 
\begin{array}{cc}
I & l^{i+1}S^{i} \\
0 & I
\end{array}
\right )H^{i}\times \tilde{H}^{i}.
$$
\end{lemma}
\begin{proof}
From \eqref{sum}, we have
\begin{align}\label{sum-2}
\ker(\mathscr{A}^{i})=\ran(\mathscr{A}^{i-1})+\left ( 
\begin{array}{cc}
I & l^{i+1}S^{i} \\
0 & I
\end{array}
\right )H^{i}\times \tilde{H}^{i}.
\end{align}
To verify that \eqref{sum-2} is a direct sum, we let 
$$
(\omega, \mu)\in \ran(\mathscr{A}^{i-1})\cap \left ( 
\begin{array}{cc}
I & l^{i+1}S^{i} \\
0 & I
\end{array}
\right )H^{i}\times \tilde{H}^{i},
$$
i.e., for some $\alpha $, $\beta$ and $h\in H^{i}$, $\tilde{h}\in \tilde{H}^{i}$, 
$$
\omega= D^{i-1}\alpha-S^{i-1}\beta, \quad \mu=\tilde{D}^{i-1}\beta, \quad \omega=h+ l^{i-1}S^{i-1}\tilde{h}, \quad \mu=\tilde{h}.
$$
Since $\tilde H^i$ represents the cohomology, it follows that $\mu=0$ and so $\beta\in\ker(\tilde D)$ and  $\omega\in H^i$.
Using the hypothesis $S\ker(\tilde D)\subset \ran(D)$ we have $\omega\in\ran(D)$ as well, and so $\omega=0$.
 \end{proof}

\subsection{From the twisted complex to the output complex}\label{step2}

In this section we prove Theorem \ref{thm:dimension} by splitting the twisted complex into two subcomplexes as outlined above.

Recall that $t^i:\mathbb{E}^i\to \tilde{\mathbb{E}}^{i-1}$ is the Moore--Penrose inverse of $s^{i-1}$, defined via \eqref{tsst}, and that $T^i=\mathrm{id}\otimes t^i$.
For future reference we establish some simple identities.
\begin{lemma}\label{lem8}
For each $i$,
\begin{gather}
\label{id1}
P_{\ker^\perp} \tilde D^{i-1} T^{i} = -T^{i+1} D^i P_\ran,\\
\label{id2}
P_{\ran^\perp} D^i P_\ran = 0,\\
\label{id3}
D^{i+1}P_\ran D^i = - D^{i+1}P_{\ran^\perp}D^i,\\
\label{id4}
P_{\ran^\perp}D^{i+1}P_{\ran^\perp} D^i = 0,\\
\label{id5}
P_{\ker} \tilde D^{i} P_{\ker} = \tilde D^i P_{\ker}.
\end{gather}
\end{lemma}
\begin{proof}
For the first identity multiply \eqref{SDDS} (with $i$ replaced by $i-1$)
on the left by $T^{i+1}$ and on the right by $T^i$ and use \eqref{tsst}.
The second identity holds because $D^i\ran(S^{i-1})\subset \ran(S^i)$, again due to \eqref{SDDS}.  The third is immediate from $D^{i+1}D^i=0$.  The left-hand side of
\eqref{id4} can be written as $P_{\ran^\perp}D^{i+1}(I-P_{\ran}) D^i$ which vanishes by \eqref{id2} and \eqref{id3}.  The identity \eqref{id5} holds because $\tilde D$ maps $\ker(S^{i-1})$ into $\ker(S^i)$ by \eqref{SDDS}.
\end{proof}

We now define a bounded linear map $\Pi^i:Y^i\to Y^i$ by
\begin{equation}\label{pidef}
\Pi^i(\omega,\mu) = 
\begin{cases}
\bigl(P_{\ran(S^{i-1})^\perp}\omega, T^{i+1}D^iP_{\ran(S^{i-1})^\perp}\omega\bigr), & 0\le i\le J,\\
\bigl(0,P_{\ker(S^i)}(\mu + \tilde D^{i-1}T^i\omega)\bigr), & J<i\le n.
\end{cases}
\end{equation}
The projections in \eqref{pidef} are defined such that the range of $\Pi^i$ is isomorphic to the output complex and the diagram commutes.
\begin{lemma}\label{lem9}
$\Pi^\bs:(Y^\bs,\mathscr{A}^\bs) \to (Y^\bs,\mathscr{A}^\bs)$ is a cochain projection.
\end{lemma}
\begin{proof}
We first show that $\Pi^\bs$ is a cochain map (commutes with $\mathscr{A}^\bs$) and then that it is projection ($(\Pi^\bs)^2=\Pi^\bs$).

To establish commutativity that $\Pi^{i+1}\mathscr{A}^i=\mathscr{A}^i\Pi^i$ for $i<J$,
we must show, in matrix notation, that
$$
\begin{pmatrix}
P_{\ran^\perp} & 0 \\ TDP_{\ran^\perp} & 0
\end{pmatrix}
\begin{pmatrix}
D & -S \\ 0 & \tilde D
\end{pmatrix}
=
\begin{pmatrix}
D & -S \\ 0 & \tilde D
\end{pmatrix}
\begin{pmatrix}
P_{\ran^\perp} & 0 \\ TDP_{\ran^\perp} & 0
\end{pmatrix}.
$$
The left-hand side simplifies to
$$
\begin{pmatrix}
P_{\ran^\perp}D & 0 \\ TDP_{\ran^\perp}D & 0
\end{pmatrix}
$$
since $P_{\ran^\perp} S = 0$.  Comparing to the right-hand side we must show
the two equations
\begin{equation}\label{twoe}
P_{\ran^\perp}D = D P_{\ran^\perp} - S TDP_{\ran^\perp}, \quad
TDP_{\ran^\perp}D = \tilde D TDP_{\ran^\perp}.
\end{equation}
In view of \eqref{tsst}, the right-hand side of the first of these equations is
$(I-P_\ran)DP_{\ran^\perp} = P_{\ran^\perp}DP_{\ran^\perp}$.
This indeed equals $P_{\ran^\perp}D$ by \eqref{id2}.
Since $i< J$ and so $S^{i+1}$ is injective, it suffices to prove the second equation in \eqref{twoe} after multiplying both sides on the left by $S$.  Using \eqref{tsst}, \eqref{id4}, \eqref{id2}, \eqref{id3}, \eqref{tsst}, and \eqref{SDDS} we get
\begin{multline*}
ST D P_{\ran^\perp} D = P_\ran D P_{\ran^\perp} D = DP_{\ran^\perp}D=
DP_{\ran^\perp}DP_{\ran^\perp}
\\
= - D P_\ran DP_{\ran^\perp} = - D S T DP_{\ran^\perp}
= S\tilde D T DP_{\ran^\perp},
\end{multline*}
as desired.  This completes the proof of commutativity for $i< J$.

Next we show commutativity for $i=J$, which comes down to
$$
\begin{pmatrix}
0 & 0 \\ P_\ker \tilde D^J (S^J)^{-1} & P_\ker
\end{pmatrix}
\begin{pmatrix}
D^J & -S^J \\ 0 & \tilde D^J
\end{pmatrix}
=
\begin{pmatrix}
D^J & -S^J \\ 0 & \tilde D^J
\end{pmatrix}
\begin{pmatrix}
P_{\ran^\perp} & 0 \\ (S^J)^{-1} D^JP_{\ran^\perp} & 0
\end{pmatrix}.
$$
This reduces to the equation
$$
P_\ker \tilde D^J (S^J)^{-1}D^J =  \tilde D^J (S^J)^{-1}D^JP_{\ran^\perp}.
$$
This is true since both sides equal $\tilde D^J (S^J)^{-1}D^J$.  Indeed
$S^{J+1} \tilde D^J (S^J)^{-1}D^J=0$ by \eqref{SDDS}, so $\tilde D^J (S^J)^{-1}D^J\in\ker(S)$,
and similarly $\tilde D^J (S^J)^{-1}D^JS^{J-1}=0$ so $\tilde D^J (S^J)^{-1}D^JP_\ran=0$.

For commutativity in the case $i>J$, we must verify that
$$
\begin{pmatrix}
0 & 0 \\ P_\ker \tilde D T & P_\ker
\end{pmatrix}
\begin{pmatrix}
D & -S \\ 0 & \tilde D
\end{pmatrix}
=
\begin{pmatrix}
D & -S \\ 0 & \tilde D
\end{pmatrix}
\begin{pmatrix}
0 & 0 \\ P_\ker \tilde D T & P_\ker
\end{pmatrix}.
$$
The top row of each product vanishes (using $SP_\ker=0$).  This leaves the equations
$$
P_\ker \tilde D T D = \tilde D P_\ker \tilde D T, \quad -P_\ker \tilde D T S + P_\ker \tilde D = \tilde DP_\ker.
$$
For the first we use that $S^{i-1}T^{i}=I$ for $i>J$, whence $D^i = D^i S^{i-1}T^i = -S^i \tilde D^{i-1} T^i$, so
\begin{multline*}
P_\ker \tilde D T D = -P_\ker \tilde D T S\tilde D T =
-P_\ker \tilde D P_{\ker^\perp} \tilde D T
\\
=
P_\ker \tilde D (I -P_{\ker^\perp}) \tilde D T =
P_\ker \tilde D P_{\ker} \tilde D T = \tilde D P_{\ker} \tilde D T.
\end{multline*}
again invoking \eqref{id5}.
For the second equation,  we rewrite the left-hand side as
$$
P_\ker\tilde D(I-P_{\ker^\perp})=P_\ker\tilde D P_{\ker} = \tilde D P_{\ker},
$$
where we have invoked \eqref{id5} in the last step.

Having established that $\Pi^\bs$ is a cochain map we now check that it is a projection, i.e., that the two matrices
$$
\begin{pmatrix}
P_{\ran^\perp} & 0 \\ TDP_{\ran^\perp} & 0
\end{pmatrix},\quad
\begin{pmatrix}
0 & 0 \\ P_\ker \tilde D^J T & P_\ker
\end{pmatrix}
$$
are idempotent.  This is immediate using the fact that $P_{\ran^\perp}$ and $P_\ker$ are projections.
\end{proof}

From the lemma, it follows directly that the twisted complex $Y^\bs=(Y^\bs,\mathscr{A}^\bs)$ splits into a direct sum of two subcomplexes, $\Pi^\bs Y^\bs$ and $(I-\Pi^\bs)Y^\bs$,
and, consequently that the $i$th cohomology space $\mathcal H^i(Y^\bs)$ is isomorphic to
the direct sum $\mathcal H^i(\Pi^\bs Y^\bs)$ and $\mathcal H^i((I-\Pi^\bs) Y^\bs)$.
We shall show (in Lemma~\ref{lem10}) that the second subcomplex, $(I-\Pi^\bs)Y^\bs$, has vanishing cohomology, and consequently that the cohomology of the complex $Y^\bs$ is isomorphic (under the map induced by $\Pi^\bs$) to the cohomology of subcomplex $\Pi^\bs Y^\bs$.
We will then show (in Lemma~\ref{lem11}) that the subcomplex $\Pi^\bs Y^\bs$ is
isomorphic, as a complex, to the ouput complex \eqref{reduced-complex}.

From the definition \eqref{pidef} of the bounded cochain projection $\Pi^\bs$, we easily identify
the subcomplexes $\Pi^\bs Y^\bs$ and $(I-\Pi^\bs) Y^\bs$:
\begin{equation}\label{PiY}
\Pi^i Y^i = 
\begin{cases}
\{\,(\omega,TD\omega)\,:\, \omega\in \ran(S^{i-1})^\perp \,\} & 0\le i\le J,\\
0 \x \ker(S^i), & J<i\le n,
\end{cases}
\end{equation}
and
\begin{equation}\label{I-PiY}
(I-\Pi^i)Y^i = 
\begin{cases}
\ran(S^{i-1})\x\tilde Z^i & 0\le i\le J,\\
\{\,(\omega,\mu-P_\ker \tilde D T\omega)\,:\, \omega\in Z^i,\ \mu\in\ker(S^i)^\perp\,\}, & J<i\le n.
\end{cases}
\end{equation}

\begin{lemma}\label{lem10}
The complex $((I-\Pi^\bs)Y^\bs,\mathscr{A}^\bs)$ is exact.
\end{lemma}
\begin{proof}
First suppose $i\le J$.  A typical element of $y$ of $(I-\Pi^i)Y^i$ can be written as
$y  = (S\beta,\mu)$ for some $\beta\in \tilde Z^{i-1}$, $\mu\in \tilde Z^i$.  If $y\in\ker(\mathscr{A}^i)$, then 
$0 = DS\beta-S\mu=-S(\tilde D \beta+\mu)$. Since $S$ is injective, this implies $\mu=-\tilde D\beta$.  Then $y=\mathscr{A}(0,-\beta)$ and $(0,-\beta)\in(I-\Pi^{i-1})Y^{i-1}$.  This establishes exactness for $i\le J$.

Now suppose that $i>J$, and let $y=(\omega,\mu-P_\ker \tilde D T\omega)$ for some
$\omega\in Z^i$, $\mu\in\ker(S)^\perp$, a typical element of $(I-\Pi^i)Y^i$. If $y\in\ker(\mathscr{A}^i)$, then
$$
D\omega = S\mu, \quad  \tilde{D}\left ( \mu-P_{\ker}\tilde{D}T \omega\right )=0.
$$
Combining \eqref{id1}, the first of these equations, and \eqref{tsst}, we get
$$
P_{\ker^{\perp}}\tilde D T \omega = -T D \omega = -TS\mu = -P_{\ker^{\perp}}\mu=-\mu,
$$
from which it follows that $\tilde D T\omega = P_{\ker}\tilde{D}T\omega-\mu$.
Therefore,
$$
\mathscr{A} (0,-T\omega)= (\omega, \mu-P_{\ker}\tilde{D}T\omega) = y
$$
and $(0,-T\omega)\in(I-\Pi^{i-1})Y^{i-1}$.
\end{proof}

Next we show that there is a simple isomorphism from the subcomplex $(\Pi^\bs Y^\bs,\mathscr{A}^\bs)$ to the output complex \eqref{reduced-complex}.

\begin{lemma}\label{lem11}
Define $\Phi^i:\Pi^i Y^i\to \Upsilon^i$ by
\begin{equation}\label{Phi}
\Phi^i(\omega,\mu) = 
\begin{cases}
\omega & 0\le i\le J,\\
\mu, & J<i\le n.
\end{cases}
\end{equation}
Then $\Phi^i$ is an isomorphism and $\mathscr{D}^i\Phi^i = \Phi^{i+1}\mathscr{A}^i$.  It follows that
$(\Upsilon^\bs,\mathscr{D}^\bs)$  is a bounded Hilbert complex and that $\Phi^\bs:(\Pi^\bs Y^\bs,\mathscr{A}^\bs) \to (\Upsilon^\bs,\mathscr{D}^\bs)$ is an isomorphism of complexes.
\end{lemma}
\begin{proof}
From the formulas \eqref{I-PiY} for  $\Pi^i Y^i$ and \eqref{defUps} for $\Upsilon^i$, it is easy to see that
$\Phi^i$ defines an isomorphism between them.  It is also straightforward from the definition of their differentials
to show that $\Phi^\bs$ is a cochain map.
\end{proof}

Combining Lemmas~\ref{lem6}--\ref{lem11} we have established Theorem~\ref{thm:dimension}. 

Finally, to prove Theorem \ref{thm:rep}, we construct cochain maps from the sum complex $(Y^{\bs}, \mathcal{D}^{\bs})$ to the output complex $(\Upsilon^{\bs}, \mathscr{D}^{\bs})$, where 
$$
\mathcal{D}^{i}=\left ( \begin{array}{cc}
D^{i} & 0\\
0 & \tilde{D}^{i}
\end{array}
\right ).
$$
The first step is to consider a cochain projection $\mathscr{Q}^{\bs}$ from the sum complex $(Y^{\bs}, \mathcal{D}^{\bs})$ to the twisted complex $(Y^{\bs}, \mathscr{A}^{\bs})$, defined by 
$$
{Q}^{i}=\left ( \begin{array}{cc}
I & K^{i}\\
0 & I
\end{array}
\right ), \quad 0\leq i\leq n.
$$
Note that $Q^{i}$ defined above is invertible.
From \eqref{SDKKD}, we get the commutativity $\mathscr{A}^{i}Q^{i}=Q^{i}\mathcal{D}^{i}$. So $Q^{\bs}$ is a cochain isomorphism. Recall that we already defined the cochain maps $\Pi^{\bs}$ \eqref{pidef} from the twisted complex $(Y^{\bs}, \mathscr{A}^{\bs})$ to its subcomplex $(\Pi^\bs Y^\bs, \mathscr{A}^{\bs})$, and $\Phi^{\bs}$ \eqref{Phi} from $(\Pi^\bs Y^\bs, \mathscr{A}^{\bs})$ to the output complex $(\Upsilon^{\bs},   \mathscr{D}^{\bs})$.  Composing the maps $Q^{\bs}$, $\Pi^\bs$ and $\Phi^{\bs}$, we obtain $\mathscr{K}^{\bs}$ defined in \eqref{def:Kmap}.  Since $Q^{\bs}$, $\Pi^\bs$ and $\Phi^{\bs}$ are all cochain maps and induce isomorphism on cohomology, we conclude that so does $\mathscr{K}^{\bs}$. This proves Theorem \ref{thm:rep}.

\ifpreprint
\subsubsection*{Acknowledgements}
The authors are grateful to Andreas \v{C}ap, Snorre Christiansen, Victor Reiner, Espen Sande, and Ragnar Winther for numerous valuable discussions related to this work.
\fi

\section*{Appendix 1. Proof of injectivity/surjectivity condition}

In this appendix we prove Lemma~\ref{lem:s-insurj}.  
Let $n>0$  and $0\le k< n$, $1\le m\le n$ be integers. The linear map $s=s^{k,m}$
is given by
$$
s:
\alt^k\R^n\otimes\alt^m\R^n \to \alt^{k+1}\R^n\otimes\alt^{m-1}\R^n
$$
by
\begin{multline}
s(v^{1}\wedge\cdots\wedge v^{k} \otimes v^{k+1}\wedge\cdots\wedge v^{k+m})
\\=\sum_{l=1}^m (-1)^{l-1} v^{k+l}\wedge v^{1}\wedge\cdots\wedge v^{k}
 \otimes v^{k+1}\wedge\cdots\wedge\widehat{v^{k+l}}\wedge\cdots\wedge v^{k+m}.
\end{multline}
Our goal is to show that $s$ is injective if $k\le m-1$ and surjective if $k\ge m-1$.

We begin with some notation.
For $n$ and $p$ natural numbers we write $[n]$ for $\{1,\ldots,n\}$ so $[n]^p$ denotes the
set of $p$-tuples of elements of $[n]$. We use standard indexing notation, so
an element $\sigma\in [n]^p$ can be written $(\sigma_1,\ldots,\sigma_p)$. The symmetric group $S_n$, the set
of permutations of $[n]$, may be viewed as a subset of $[n]^n$.  If also $0\le k\le p$, we define
$$
X(n,p,k) = \{\sigma\in [n]^p:\sigma_1<\cdots<\sigma_k,\ \sigma_{k+1}<\cdots<\sigma_p\},
$$
the set of $p$-tuples which are strictly increasing in the first $k$ indices and in the remaining $p-k$ indices.
To each $\sigma\in [n]^p$ we may associate
$$
dx^\sigma:= dx^{\sigma_1}\wedge\cdots\wedge dx^{\sigma_{k}}\otimes dx^{\sigma_{k+1}}\wedge\cdots\wedge dx^{\sigma_{n}}\in \alt^k\R^n\otimes\alt^{p-k}\R^n.
$$
where the $dx^i$ are the usual basis elements of the dual space of $\R^n$. The $dx^\sigma$ for $\sigma\in X(n,p,k)$ then form the standard basis for $\alt^k\R^n\otimes\alt^{p-k}\R^n$.

Turning to the proof of Lemma~\ref{lem:s-insurj}, we first consider the case $m=n-k$.  In this case,
$$
s:\alt^k\R^n\otimes\alt^{n-k}\R^n\to\alt^{k+1}\R^n\otimes\alt^{n-k-1}\R^n
$$
and we wish to show injectivity for $n-2k-1\ge0$ and surjectivity for $n-2k-1\le0$.
Given a subset $I\subset[n]$ of cardinality $k$, let $\sigma\in S_n\cap X(n, n, k)$ be
the unique element for which $\{\sigma_1,\ldots,\sigma_k\} = I$ and set
$\omega(I) = \sign(\sigma)dx^\sigma \in \alt^k\R^n\otimes\alt^{n-k}\R^n$.
Let $W(n,k)$ denote the subspace of $\alt^k\R^n\otimes\alt^{n-k}\R^n$ spanned by the elements $\omega(I)$
for all subsets $I$ of $[n]$ of cardinality $k$.
It then follows from the definition of $s$ that
$$
s\omega(I) = (-1)^{k} \sum_{j\in[n]\setminus I}\omega(I\cup\{j\}).
$$
In particular, $sW(n,k)\subset W(n,k+1)$.  We define an inner product on each $W(n,k)$ by declaring the basis elements
$\omega(I)$ to be orthonormal.  Then the adjoint $s^*:W(n,k+1)\to W(n,k)$ is given by
$$
s^*\omega(J) = (-1)^{k} \sum_{j\in J} \omega(J\setminus\{j\}), \quad J\subset[n], \ \#J=k+1.
$$
The next result shows the desired injectivity and surjectivity in the case $m=n-k$, but only for
the restriction of $s$ to a map from $W^k$ to $W^{k+1}$.
\begin{lemma}\label{lem:injsur-1}
If $n-2k-1\ge 0$, then $s:W(n,k)\to W(n,k+1)$ is injective.  If $n-2k-1\le0$, then it is surjective.
\end{lemma}
\begin{proof}
Let $J,K$ be subsets of $[n]$ of cardinality $k$. Then
$$
\langle s\omega(J),s\omega(K)\rangle =
\begin{cases}
n-k, & J=K,\\
1, & \#J\cap K=k-1,\\
0,& \text{else},
\end{cases}
$$
and
$$
\langle s^*\omega(J),s^*\omega(K)\rangle =
\begin{cases}
k, & J=K,\\
1, & \#J\cap K=k-1,\\
0,& \text{else}.
\end{cases}
$$
It follows that
$$
\langle s\omega(J),s\omega(K)\rangle =\langle s^*\omega(J),s^*\omega(K)\rangle + (n-2k)\langle \omega(J),\omega(K)\rangle,
$$
and, by bilinearity, that
\begin{equation}\label{sident}
\langle s\rho,s\tau\rangle =
\langle s^*\rho,s^*\tau\rangle + (n-2k)\langle \rho,\tau\rangle,\quad \rho,\tau\in W(n,k).
\end{equation}
Taking $\tau=\rho$ and assuming that $n-2k-1\ge 0$, we see that
$s\rho=0$ implies $\rho=0$, so $s$ is injective as claimed.

If we replace $k$ by $k+1$ in \eqref{sident} and assume that $n-2k-1\le 0$, the same argument implies that
$s^*:W(n,k+1)\to W(n,k)$ is injective, and consequently that $s$ is surjective.
\end{proof}

Now we return to general $n\ge1$, $0\le k<n$, $1\le m\le n$, and the map $s$ acting on all
of $\alt^k\R^n\otimes\alt^{m}\R^n$.  To prove surjectivity, assuming $k\ge m-1$, we must show that $s$ maps onto
all of $\alt^{k+1}\R^n\otimes \alt^{m-1}\R^n$.  For this it is enough to take an element
of the form
$$
\rho = v^1\wedge\cdots\wedge v^{k+1}\otimes v^{k+2}\wedge\cdots\wedge v^{k+m}
$$
with the $v^i$ belonging to the dual of $\R^n$, and show that $\rho$ is in the range of $s$.

Let $p=m+k$ and define a linear map from the dual space of $\R^p$ to that of $\R^n$
by $T dx^i = v^i$, $i=1,\ldots,p$.
Then $T$ induces a linear map
$$
T_*:\alt^k\R^p\otimes\alt^m\R^p \to \alt^k\R^n\otimes\alt^m\R^n
$$
given by
$$
T_*(u^1\wedge\cdots\wedge u^k\otimes u^{k+1}\wedge\cdots\wedge u^{k+m}) =
(Tu^1\wedge\cdots\wedge Tu^k\otimes Tu^{k+1}\wedge\cdots\wedge Tu^{k+m}).
$$
Clearly, $T_*s = sT_*$ and, letting
$$
\omega = dx^1\wedge\cdots\wedge dx^{k+1}\otimes dx^{k+2}\wedge\cdots\wedge dx^{k+m},
$$
we have $T_*\omega = \rho$.  Since $\omega\in W(n,k+1)$, the preceding lemma insures that $\omega=s \mu$ for some $\mu\in W(n,k)\subset \alt^k\R^p\otimes\alt^m\R^p$.  Therefore
$$
\rho = T_*\omega=T_*s\mu= s(T_*\mu).
$$
This completes the proof of surjectivity.

We now prove the injectivity for general $n$, $k$, and $m$, continuing to write $p=m+k$.
For $\sigma\in X(n,p,k)$ let $\tilde\sigma\in[n]^p$ denote the tuple obtained from $\sigma$ by taking its
entries in non-decreasing order.  For example, if $\sigma=(2,3,1,2)\in X(3,4,2)$ (so increasing on its first 2 and last 2 indices), then $\tilde\sigma = (1,2,2,3)$.  Then we have the direct sum decomposition
$$
\alt^k\R^n\otimes\alt^m\R^n = \bigoplus_{J\in[n]^p} Y(n,p,k,J),
$$
where
$$
Y(n,p,k,J)=\spn\{dx^\sigma:\sigma\in X(n,p,k),\ \tilde\sigma=J\}.
$$
Of course, there is a similar decomposition for $\alt^{k-1}\R^n\otimes\alt^{m+1}\R^n$. The two decompositions are
compatible with $s$, in the sense that $sY(n,p,k,J)\subset Y(n,p,k+1,J)$ for the same $J$.
It follows that it is enough to prove that $s$ is injective when restricted to each of the spaces
$Y(n,p,k,J)$, $J\in[n]^p$.  The $p$-tuple $J$ consists of entries which appear only once and entries which appear twice.
Let $l$ be the number of repeated entries, so that there are $q:=p-2l$ non-repeated entries.  Without loss of generality, we may assume that the non-repeated entries are $1,\ldots,q$ and the repeated entries $q+1,\ldots,q+l$, i.e.,
$$
J = (1,2,\ldots,q,q+1,q+1, q+2, q+2,\ldots,q+l,q+l).
$$
The space $S_{p-2l}\cap X(p-2l,p-2l,k-l)$ consists of permutations of $[p-2l]$ which are increasing in their first $k-l$ and last $m-l$ indices.  If $\rho$ belongs to this space, we define $Q\rho$ as the $p$-tuple
$$
Q\rho = (\rho_1,\ldots,\rho_{k-l},q+1, q+2,\ldots,q+l, \rho_{k-l+1}, \ldots, \rho_{p-2l},q+1, q+2,\ldots,q+l). 
$$
This defines a bijection of $S_{p-2l}\cap X(p-2l,p-2l,k-l)$ onto 
$\{\sigma\in X(n,p,k),\ \tilde\sigma=J\}$.
Now we consider the spaces spanned by the basis functions $dx^\sigma$ where $\sigma$ varies in one of these two bijective sets.
These spaces are precisely $W(p-2l,k-l)$ and $Y(n,p,k,J)$, respectively, and the bijection just established induces an isomorphism
$F:Y(n,p,k,J)\to W(p-2l,k-l)$, given by
$$
dx^{Q\rho} \mapsto dx^\rho.
$$
It is easy to see that $Fs=sF$.  If $\omega\in Y(n,p,k,J)$ and $s\omega=0$, then $F\omega\in W(p-2l,k-l)$ and $sF\omega=0$,
so, by Lemma~\ref{lem:injsur-1}, $F\omega=0$, so $\omega=0$.  This completes the proof.

\bibliographystyle{spmpsci}
\bibliography{ElasticityComplex}

\end{document}